\documentclass{aims}
\usepackage{amsmath}
\usepackage{amssymb}
\usepackage{paralist}
\usepackage{graphics} 
\usepackage{epsfig} 
\usepackage{graphicx}  \usepackage{epstopdf}
\usepackage{mathrsfs}
\usepackage{cancel}
\usepackage{xcolor}
\usepackage{tikz-cd}
\usepackage{caption}
\usepackage{subcaption}

 \usepackage[colorlinks=true]{hyperref}
\hypersetup{urlcolor=blue, citecolor=red}

\def\currentvolume{X}
 \def\currentissue{X}
  \def\currentyear{200X}
   \def\currentmonth{XX}
    \def\ppages{X--XX}

\newtheorem{theorem}{Theorem}[section]

\newtheorem{lemma}[theorem]{Lemma}
\newtheorem{proposition}{Proposition}

\theoremstyle{definition}
\newtheorem{definition}[theorem]{Definition}
\newtheorem{remark}{Remark}

\newtheorem{example}{Example}

\DeclareMathOperator{\tr}{tr}

\title[Nonholonomic Volume] 
      {Existence of Invariant Volumes in Nonholonomic Systems Subject to Nonlinear Constraints}

\author[William Clark and Anthony Bloch]{}

\subjclass{Primary: 70F25, 37C40; Secondary: 70G45.}
 \keywords{Geometric Mechanics, Nonholonomic Systems, Invariant Volumes}

 \email{wac76@cornell.edu}
 \email{abloch@umich.edu}

 \thanks{The first author is supported by NSF grant DMS-1645643. The second
 author is supported by NSF grants DMS-1613819 DMS-2103026 and AFOSR grant 77219283}

\thanks{$^*$ Corresponding author: W. Clark}

\begin{document}
\maketitle

\centerline{\scshape William Clark$^*$}
\medskip
{\footnotesize
 \centerline{Department of Mathematics, Cornell University}
   \centerline{301 Tower Rd, Ithaca, NY, USA}
} 

\medskip

\centerline{\scshape Anthony Bloch}
\medskip
{\footnotesize
	\centerline{Department of Mathematics, University of Michigan}
	\centerline{530 Church Street, Ann Arbor, MI, USA}
}
\bigskip


\begin{abstract}
We derive conditions for a nonholonomic system subject to nonlinear constraints (obeying Chetaev's rule) to preserve a smooth volume form. When applied to affine constraints, these conditions dictate that a basic invariant density exists if and only if a certain 1-form is exact and a certain function vanishes (this function automatically vanishes for linear constraints). Moreover, this result can be extended to geodesic flows for arbitrary metric connections and the sufficient condition manifests as integrability of the torsion. As a consequence, volume-preservation of a nonholonomic system is closely related to the torsion of the nonholonomic connection. Examples of nonlinear/affine/linear constraints are considered.
\end{abstract}

\section{Introduction}
An invariant volume is a powerful tool for understanding the asymptotic nature of a dynamical system. In particular, it is a well-known fact that unconstrained mechanical systems are volume-preserving. The case of nonholonomic systems is more nuanced as these systems generally fail to preserve the symplectic form (which follows from the fact that nonholonomic systems are \textit{not} variational) and hence, the induced volume form. This makes the study of invariant volumes in nonholonomic systems nontrivial. A famous example of this is the Chaplygin sleigh; this system, although energy-preserving, experiences ``dissipation'' which results in asymptotic stability (cf. \cite{zenkov1997} for a general discussion on stability of nonholonomic systems or \cite{ruina1998} for an interpretation via impact systems).

The existence of an invariant volume for a nonholonomic system offers two key insights. The first is the usual case in dynamical systems where an invariant measure allows for the use of the Birkhoff Ergodic Theorem (cf. e.g. 4.1.2 in \cite{katok1995introduction}) as well as for recurrence {(with the caveat of the volume being finite)}. The other is unique to nonholonomic systems; even though nonholonomic systems are \textit{not} Hamiltonian, ``nonholonomic systems which do preserve volume are in a quantifiable sense closer to Hamiltonian systems than their volume changing counterparts,'' \cite{fernandez2009hamiltonization} (see also \cite{BaGa2012,BoBi2015,BoBoMa2011,AvBor2013,VVKo1988}). Therefore, being able to find an invariant measure for a nonholonomic system allows for ergodic-like understanding of its asymptotic behavior and can provide a way to ``Hamiltonize'' a nonholonomic system ({although most} systems with an invariant volume still cannot be Hamiltonized, cf. \cite{Jovanovic_2019}).

There exists an abundance of research into finding invariant volumes for nonholonomic systems where the constraints are linear and symmetries are present: Chaplygin systems are studied in, e.g. \cite{AvBor2013,cantrijn2002,garcia2020,ILIYEV1985295, Koiller1992, monforte2004geometric, neimark1972dynamics}, Euler-Poincar\'{e}-Suslov systems are studied in, e.g. \cite{bloch2008nonholonomic, Jovanovic_1998}, systems with internal degrees of freedom are studied in, e.g. \cite{bloch2008nonholonomic, bloch2009quasivelocities, zenkov2003invariant}, and \cite{federovnaranjo} studies the case of symmetric kinetic systems where the dimension assumption does not hold. Related work on asympotic dynamics may be found in \cite{YoMo2020}. There are also results demonstrating that no invariant volumes exist, e.g. \cite{AvBor2005,nonexistenceellipsoid}.

The contribution of this work is to consider conditions for the existence of invariant volume in nonholonomic systems such that
\begin{enumerate}
	\item the constraints need not be linear/affine in the velocities,
	\item the analysis can be carried out on the ambient manifold rather than resorting to local coordinates, and
	\item absolutely no symmetry assumptions are used.
\end{enumerate}
Such an approach seems to be new.

Our main result is an existence condition for an invariant measure for a hyper-regular Lagrangian system  with an admissible set of nonlinear nonholonomic constraints  obeying 
Chetaev's rule. In the linear case the existence of a basic invariant density reduces to checking whether or not a certain 1-form can be made to be exact, \cite{blackall1941}, and this special case is similar to results in \cite{federovnaranjo,garcia2020}.

Preliminaries on nonholonomic systems are presented in Section \ref{sec:prelim}. 
Section \ref{sec:gnvf} presents the construction of the ``extended nonholonomic vector fields'' on the entire cotangent space $T^*Q$ which restricts to the nonholonomic vector field on $M$. Although there does not exist a canonical choice for the extended vector field, they will agree once restricted. An advantage of this approach is that one is able to design extended vector fields which ensure that the constraint {manifold} is an exponentially stable invariant manifold. 
The divergence calculation for a nonholonomic system is performed in Section \ref{sec:volume} via extended nonholonomic vector fields. 
The main results, Theorem \ref{th:main_proved} and Theorem \ref{th:affine_basic}, are proved in Section \ref{sec:cohom} (cf. Theorems \ref{th:main_proved} and \ref{th:affine_basic}). Additionally, it is shown that the invariant volume does not depend on the choice of extended nonholonomic vector field chosen and is unique {up to a} constant of motion. 
Section \ref{sec:connection} presents an interpretation of the linear constraint case to the torsion of the nonholonomic connection, which seems to be a novel observation. 
This paper concludes with examples in Section \ref{sec:examples}.


This paper is a continuation of the work done in \cite{clarkthesis} and, as such, many of the results below can be found there.

\section{Preliminaries and Notation}\label{sec:prelim}
\subsection{Unconstrained Mechanics}
We will first briefly cover the case of unconstrained mechanical systems before discussing nonholonomic systems. A smooth (finite-dimensional, $\mathrm{dim}(Q)=n$) manifold $Q$ is called the \textit{configuration space}, the tangent bundle $TQ$ is called the \textit{state space}, and the cotangent bundle $T^*Q$ is called the \textit{phase space}. The bundle projection maps will be denoted by $\tau_Q:TQ\to Q$ and $\pi_Q:T^*Q\to Q$. 
\subsubsection{Lagrangian Systems}
Lagrangian systems are described by a smooth real-valued function on the state space called the Lagrangian, $L:TQ\to Q$. The dynamics generated from the Lagrangian function arise from Hamilton's principle and are given by the Euler-Lagrange equations,
\begin{equation*}
	\frac{d}{dt}\frac{\partial L}{\partial \dot{q}} - \frac{\partial L}{\partial q} = 0.
\end{equation*}
\subsubsection{Hamiltonian Systems}
In constrast to Lagrangian systems, Hamiltonian systems evolve on the phase space. For a given Lagrangian, denote $\mathbb{F}L:TQ\to T^*Q$ as its fiber derivative.
\begin{definition}
	A Lagrangian, $L:TQ\to \mathbb{R}$, is \textit{hyperregular} if the fiber derivative is a global diffeomorphism. Furthermore, a Lagrangian is \textit{natural} if it has the form
	\begin{equation*}
		L(v) = \frac{1}{2}g(v,v) - \tau_Q^*V(v),
	\end{equation*}
	where $g$ is a Riemannian metric on $Q$ and $V\in C^\infty(Q)$.
\end{definition}
Throughout this work, all Lagrangians will be assumed to be hyperregular. Moreover, when affine constraints are considered, the Lagrangian will be assumed to be natural. 

For a given Lagrangian, we define the Hamiltonian $H:T^*Q\to\mathbb{R}$ via the \textit{Legendre transform}:
\begin{equation*}
	H(p) = \langle p,v\rangle - L(v), \quad p = \mathbb{F}L(v).
\end{equation*}
Let $\omega = dq^i\wedge dp_i$ be the canonical symplectic form on $T^*Q$. The resulting Hamiltonian vector field arising from the Hamiltonian $H$ and the symplectic form $\omega$ is denoted by $X_H$ and is given by
\begin{equation*}
	i_{X_H}\omega = dH,
\end{equation*}
where $i_X\omega = \omega(X,\cdot)$ is the contraction. An important feature of Hamiltonian vector fields is that they preserve the symplectic form, and thus are volume-preserving.
\begin{theorem}[Liouville's Theorem]\label{th:liouville}
	Hamiltonian dynamics preserve the symplectic form and, additionally, preserve the volume form $\omega^n$.
\end{theorem}
Liouville's theorem does not apply to general nonholonomic systems as they are generally not symplectic. Determining when an analogous version of this theorem applies to constrained systems is the main goal of this work.
\subsection{Constrained Dynamics}
For unconstrained Lagrangian systems, motion is allowed to occur in the entire state space, $TQ$. For constrained systems, motion is restricted to lie on a constraint manifold $N\subset TQ$ and we tacitly assume that $\tau_Q(N)=Q$. For the purposes of this work, we will not differentiate between nonholonomic and holonomic constraints as the machinery for the former will suffice for the latter. 

Let $N\subset TQ$ be the constraint submanifold which is, locally, described via the zero level-set of a regular function $G:TQ\to\mathbb{R}^k$, $k<n=\dim(Q)$. Let us denote each component of this function by $\Psi^\alpha$, $\alpha=1,\ldots,k$, i.e.
\begin{equation*}
	N = \bigcap_{\alpha=1}^k \, (\Psi^\alpha)^{-1}(0).
\end{equation*}
To transfer to the Hamiltonian point of view, we assume that the Lagrangian, $L:TQ\to\mathbb{R}$, is hyper-regular. The constraints on the cotangent bundle become
\begin{equation*}
	\Phi^\alpha = (\mathbb{F}L^{-1})^*\Psi^\alpha, \quad M = \mathbb{F}L(N) = \bigcap_{\alpha=1}^k \, \left( \Phi^\alpha\right)^{-1}(0).
\end{equation*}
These constraints are admissible if the following matrix is non-singular \cite{monforte2004geometric}.
\begin{equation*}
	m^{\alpha\beta} = \mathcal{C}^*d\Phi^\beta\left( X_{\Phi^\alpha}\right), \quad (m_{\alpha\beta}) = (m^{\alpha\beta})^{-1},
\end{equation*}
where $X_{\Phi^\alpha}$ is the Hamiltonian vector field generated by $\Phi^\alpha$ and $\mathcal{C}:T(T^*Q)\to T(T^*Q)$ is the $\mathbb{F}L$-related almost-tangent structure, {explained below}.


The constrained Euler-Lagrange equations become modified via Chetaev's rule (which is not necessarily the {physically} correct procedure, cf. \cite{marle} for a discussion), which will provide equivalent results to those in \cite{deLeon1997} where the ``almost-tangent'' structure of the tangent bundle is utilized. For Lagrangian systems, Chetaev's rule states that if we have the nonlinear constraints $\Psi^\alpha$, then the constraint forces have the following form
\begin{equation}\label{eq:Lagrange_Chetaev}
	\frac{d}{dt}\left( \frac{\partial L}{\partial \dot{q}}\right) - \frac{\partial L}{\partial q} = \lambda_\alpha\cdot \mathcal{S}^*d\Psi^\alpha,
\end{equation}
where $\mathcal{S}:T(TQ)\to T(TQ)$ is {a (1,1)-tensor} called the almost-tangent structure and, in local coordinates, is given by 
\begin{equation*}
	S = \frac{\partial}{\partial\dot{q}^i}\otimes dq^i.
\end{equation*}
If, rather than being general nonlinear, the constraints are affine in the velocities,
\begin{equation*}
	\Psi^\alpha(v) = \eta^\alpha(v) + \tau_Q^*\xi^\alpha(v),
\end{equation*}
where $\eta^\alpha\in\Omega^1(Q)$ are 1-forms and $\xi^\alpha\in C^\infty(Q)$ are functions, then \eqref{eq:Lagrange_Chetaev} becomes
\begin{equation*}
	\frac{d}{dt}\left( \frac{\partial L}{\partial \dot{q}}\right) - \frac{\partial L}{\partial q} = \lambda_\alpha\cdot \tau_Q^*\eta^\alpha.
\end{equation*}
However, this work will focus on the Hamiltonian formalism. The constraint manifold, $M\subset T^*Q$, is locally described by the joint level-set of the functions $\Phi^\alpha:T^*Q\to\mathbb{R}$ where $\Phi(p) = \Psi\circ\mathbb{F}L^{-1}(p)$. Moving \eqref{eq:Lagrange_Chetaev} to the cotangent side, the constrained Hamiltonian equations of motion become
\begin{equation}\label{eq:Hamiltonian_Chetaev}
	i_{X_H^M}\omega|_M = dH|_M + \lambda_\alpha\cdot \mathcal{C}^*d\Phi^\alpha|_M,
\end{equation}
where $\mathcal{C}:T(T^*Q)\to T(T^*Q)$ is the $\mathbb{F}L$-related almost-tangent structure, i.e. if $L$ is natural, then
\begin{equation*}
	\mathcal{C}^*\left( \alpha_idq^i + \beta^jdp_j\right) = g_{ij}\beta^jdq^i.
\end{equation*}
\subsection{Notation}
Throughout this work, $L:TQ\to\mathbb{R}$ will be a hyperregular Lagrangian with corresponding Hamiltonian $H:T^*Q\to\mathbb{R}$. The constraint manifold will be called $N\subset TQ$ or $M=\mathbb{F}L(N)\subset T^*Q$. These submanifolds will be locally described by the joint zero level-set of a collection of smooth functions,
\begin{equation*}
	N = \bigcap_{\alpha=1}^k \, (\Psi^\alpha)^{-1}(0),\quad M = \bigcap_{\alpha=1}^k\, (\Phi^\alpha)^{-1}(0).
\end{equation*}

If the constraints are affine in the velocities/momentum, the Lagrangian will be assumed to be natural with Riemannian metric $g$. In this case, the constraining functions will have the form
\begin{equation*}
	\Psi^\alpha(v) = \eta^\alpha(v) + \tau_Q^*\xi^\alpha(v), \quad
	\Phi^\alpha(p) = P(W^\alpha)(p) + \pi_Q^*\xi^\alpha(p),
\end{equation*}
where $\eta^\alpha$ are 1-forms, $\xi^\alpha$ are functions, $W^\alpha = \mathbb{F}L^{-1}(\eta^\alpha)$ are vector fields, and $P(W^\alpha)$ is the vector field's momentum
\begin{equation*}
	P(W^\alpha)(p) = \langle p, W^\alpha(\pi_Q(p))\rangle.
\end{equation*}
In this case of affine constraints, $\mathcal{D}\subset TQ$ will be the distribution
\begin{equation*}
	\mathcal{D}_q = \bigcap_{\alpha=1}^k\, \ker \eta^\alpha_q,
\end{equation*}
and $\mathcal{D}^*=\mathbb{F}L(\mathcal{D})\subset T^*Q$.

Finally, the nonholonomic vector fields with Hamiltonian $H$ and constraint manifold $M$ will be denoted by $X_H^M$.

\section{Extended Nonholonomic Vector Fields}\label{sec:gnvf}
Given a constraint manifold, $M\subset T^*Q$, we can determine the nonholonomic vector field, $X_H^M\in\mathfrak{X}(M)$ via \eqref{eq:Hamiltonian_Chetaev}. Commonly local, noncanonical, coordinates are chosen for $M$ (cf. \S 5.8 in \cite{bloch2008nonholonomic} and \cite{VANDERSCHAFT1994225}).  
However, we will instead work with a tubular neighborhood (for ease, this will be written as the entire manifold $T^*Q$) and define an \textit{extended} vector field $X_H^{ext}\in\mathfrak{X}(T^*Q)$ such that $X_H^{ext}|_{M} = X_H^M$.
This section outlines an intrinsic (albeit non-unique) way to determine such a vector field.
\begin{definition}
	For a given constraint submanifold $M\subset T^*Q$, a \textit{realization} of $M$ is an ordered collection of functions {$\mathscr{C}$} $:=\{\Phi^\alpha:T^*Q\to\mathbb{R}\}$ such that zero is a regular value of $G=\Phi^1\times\ldots\times \Phi^k$ and
	\begin{equation*}
	M = \bigcap_{\alpha}\, (\Phi^\alpha)^{-1}(0).
	\end{equation*}
	If the functions $\Phi^\alpha$ are affine in momenta, i.e. $\Phi^\alpha = P(X^\alpha)+\pi_Q^*f^\alpha$, then the realization is called \textit{affine}.
\end{definition}
\begin{remark}
	In the case where the Lagrangian is natural (which provides a Riemannian metric on $Q$) and the constraint submanifold is affine, we can choose the realization to be affine: 
	$$\mathscr{C} = \{P(W^1)+\pi_Q^*\xi^1,\ldots,P(W^k)+\pi_Q^*\xi^k\},$$ 
	where $W^i = \mathbb{F}L^{-1}\eta^i=(\eta^i)^\sharp$.
\end{remark}
By replacing $M$ with a realization $\mathscr{C}$, we can extend the nonholonomic vector field to a vector field on $T^*Q$ such that $\Phi^\alpha$ are first integrals.
Recall that the form of the nonholonomic vector field is $i_{X_H^M}\omega = dH + \lambda_\alpha\mathcal{C}^*d\Phi^\alpha$. We construct the extended nonholonomic vector field, {$\Xi_H^{\mathscr{C}}$}, by requiring that:
\begin{itemize}
	\item[(NH.1)] $i_{\Xi_H^\mathscr{C}}\omega = dH + \lambda_\alpha\mathcal{C}^*d\Phi^\alpha$ for smooth functions $\lambda_\alpha:T^*Q\to\mathbb{R}$, and
	\item[(NH.2)] $\mathcal{L}_{\Xi_H^\mathscr{C}}\Phi^\alpha = 0$ for all $\Phi^\alpha\in\mathscr{C}$.
\end{itemize}
Under reasonable compatibility assumptions on $\mathscr{C}$ (cf. \S3.4.1 in \cite{monforte2004geometric}, see Definition \ref{def:mass_matrix} below), such a vector field exists and is unique. However, given two different realizations, $\mathscr{C}$ and $\mathscr{C}'$, of the same constraint manifold $M$, it is not generally true that $\Xi_{H}^\mathscr{C} = \Xi_H^{\mathscr{C}'}$, \textit{however} $\Xi_{H}^\mathscr{C}|_{M} = \Xi_H^{\mathscr{C}'}|_{M}$. 

\begin{remark}
	The constraint manifold is given by the joint zero level-sets of the $\Phi^\alpha$ while the realization provides additional irrelevant information off of the constraint manifold. This is why $\Xi_H^\mathscr{C}\ne \Xi_H^{\mathscr{C}'}$ but they agree once restricted as will be proved in Proposition \ref{prop:agree_NH} below.
\end{remark}

\begin{definition}\label{def:mass_matrix}
	For a realization $\mathscr{C} = \{\Phi^1,\ldots,\Phi^k\}$, the constraint mass matrix, $(m^{\alpha\beta})$, is given by ``orthogonally'' pairing the constraints,
	\begin{equation*}
		m^{\alpha\beta} = \mathcal{C}^*d\Phi^\alpha\left( X_{\Phi^\beta}\right).
	\end{equation*}
	The realization is \textit{admissible} if this matrix is non-singular on a tubular neighborhood of $M$. The inverse matrix will be denoted by $(m_{\alpha\beta}) = (m^{\alpha\beta})^{-1}$.
\end{definition}
When the constraints are affine (along with a natural Lagrangian), the constraint mass matrix becomes
\begin{equation*}
	m^{\alpha\beta} = g(W^\alpha,W^\beta) = \eta^\alpha(W^\beta),
\end{equation*}
and admissibility of the constraints is equivalent to the constraints being linearly independent.

We can now write down a formula for $\Xi_H^\mathscr{C}$.
Using (NH.1) and (NH.2), we get that (where $\{\cdot,\cdot\}$ is the standard Poisson bracket)
\begin{equation*}
	\begin{split}
		\mathcal{L}_{\Xi_H^\mathscr{C}}\Phi^\beta &= i_{X_{\Phi^\beta}}\omega (\Xi_H^\mathscr{C}) \\
		&= -i_{\Xi_H^\mathscr{C}}\omega(X_{\Phi^\beta}) \\
		&= -dH\left( X_{\Phi^\beta}\right) - \lambda_\alpha \mathcal{C}^*d\Phi^\alpha(X_{\Phi^\beta}) \\
		&= \left\{ \Phi^\beta,H \right\} - \lambda_\alpha \mathcal{C}^*d\Phi^\alpha(X_{\Phi^\beta}) = 0.
	\end{split}
\end{equation*}
This implies that $\left\{ \Phi^\beta, H \right\} = m^{\alpha\beta}\lambda_\alpha$.
Due to the constraint mass matrix being nondegenerate, the multipliers are uniquely determined and the extended nonholonomic vector field is determined by
\begin{equation}\label{eq:nh_global}
	i_{\Xi_H^{\mathscr{C}}}\omega = dH - m_{\alpha\beta}\left\{H,\Phi^\alpha\right\} \mathcal{C}^*d\Phi^\beta
\end{equation}
\begin{definition}
	The 1-form on $T^*Q$ given by
	\begin{equation*}
	{\nu_H^\mathscr{C}} := dH - m_{\alpha\beta}\left\{H,\Phi^\alpha\right\} \mathcal{C}^*d\Phi^\beta,
	\end{equation*}
	is called the \textit{nonholonomic 1-form} with respect to the realization $\mathscr{C}$.
\end{definition}
\begin{proposition}\label{prop:agree_NH}
	Given two different realization, $\mathscr{C}$ and $\tilde{\mathscr{C}}$, the extended nonholonomic vector fields given by \eqref{eq:nh_global} agree on $M$.
\end{proposition}
\begin{proof}
	Let $\mathscr{C} = \{\Phi^\alpha\}$ and $\tilde{\mathscr{C}} = \{\tilde{\Phi}^\alpha\}$ be two different realizations of the same manifold $M$ and let $x\in M$. Their differentials span the annihilator,
	\begin{equation*}
		\mathrm{span}\left\{ d\Phi^\alpha_x\right\} = \mathrm{span}\left\{d\tilde{\Phi}^\alpha_x\right\} = \mathrm{Ann}(T_xM).
	\end{equation*}
	As the map $\mathcal{C}:T(T^*Q)\to T(T^*Q)$ is linear, we have
	\begin{equation*}
		\mathrm{span}\left\{ \mathcal{C}^*d\Phi^\alpha_x \right\} = \mathrm{span}\left\{ \mathcal{C}^*d\tilde{\Phi}^\alpha_x \right\}.
	\end{equation*}
	As the multipliers are unique, we must have
	\begin{equation*}
		m_{\alpha\beta}\left\{H,\Phi^\alpha\right\}\mathcal{C}^*d\Phi^\beta|_M = \tilde{m}_{\alpha\beta}\left\{H, \tilde{\Phi}^\alpha\right\}\mathcal{C}^*d\tilde{\Phi}^\beta|_M,
	\end{equation*}
	and therefore the resulting extended nonholonomic vector fields must agree on $M$.
\end{proof}
\begin{remark}
	A consequence of Proposition \ref{prop:agree_NH} is that this procedure
	of extending the nonholonomic vector field
	is still valid when $M$ is not globally defined as the level-set of functions. Suppose that $U,V\subset M$ are two neighborhoods characterized by 
	the zero level-sets of functions $\Phi,\phi:T^*Q\to\mathbb{R}$. Then $\Phi$ and $\phi$ both define different realizations on $U\cap V$. However, their corresponding vector fields agree on the intersection. 
	This observation will make the constructions requires for Theorem \ref{th:main_proved} well-defined even if $M$ cannot be globally defined as a level set.
\end{remark}
\subsection{Stabilizing Extended Vector Fields}
The extended nonholonomic vector field, $X_H^M$, is a vector field on the ambient space $T^*Q$ and has $M$ as an invariant manifold. As such, an integral curve of $X_H^M$ is only a valid trajectory if its initial condition lies in $M$. In practice using numerical techniques, round off errors can result in trajectories drifting away from $M$. To counter this effect, the condition (NH.2) can be replaced to include something in the spirit of a ``control Lyapunov function'' \cite{isidori1995}. This modification will only be explored in this section while the rest of this paper will be concerned with the unmodified approach explained above.

Let the modified feedback nonholonomic criteria be:
\begin{enumerate}
	\item[(NH.1)] $i_{\Xi_H^\mathscr{C}}\omega = dH + \lambda_\alpha \mathcal{C}^*d\Phi^\alpha$, for smooth functions $\lambda_\alpha:T^*Q\to \mathbb{R}$, and
	\item[(NH.2')] $\mathcal{L}_{\Xi_H^\mathscr{C}}\Phi^\alpha = -\kappa^\alpha \Phi^\alpha$ for some positive constants $\kappa^\alpha$, 
\end{enumerate}
where there is no summation in (NH.2').

This approach has the advantage that if $x_0\not\in M$, its trajectory exponentially approaches $M$, i.e. suppose that $\varphi_t$ is the flow of the extended feedback nonholonomic vector field, then
\begin{equation*}
	\Phi^\alpha\left( \varphi_t(x_0)\right) = e^{-\kappa^\alpha}\Phi^\alpha\left(x_0\right) \to 0.
\end{equation*}
Mimicking the derivation of \eqref{eq:nh_global}, we arrive at
\begin{equation}\label{eq:Stable_Global}
	{
	i_{\Xi_H^\mathscr{C}}\omega = dH - m_{\alpha\beta}\left( \left\{ H, \Phi^\alpha \right\} -  \kappa^\alpha \Phi^\alpha\right)\cdot \mathcal{C}^*d\Phi^\beta.
	}
\end{equation}
\begin{example}[Simple Pendulum]
	Consider the simple pendulum viewed as a constrained system in $\mathbb{R}^2$. The Hamiltonian and constraint are
	\begin{equation*}
		H = \frac{1}{2m}\left( p_x^2+p_y^2\right) + mgy, \quad \Phi = xp_x+yp_y=0,
	\end{equation*} 
	where $m$ is the mass of the pendulum and $g$ is the acceleration due to gravity. The equations of motion from \eqref{eq:Stable_Global} are
	\begin{equation}\label{eq:pendulum_stable}
		\begin{split}
		\dot{x} &= \frac{1}{m}p_x,\quad \dot{y} = \frac{1}{m}p_y, \\ 
		\dot{p}_x &= \frac{x}{x^2+y^2}\left[ mgy - \frac{1}{m}(p_x^2+p_y^2) - \kappa\left( xp_x+yp_y\right) \right], \\
		\dot{p}_y &= \frac{y}{x^2+y^2}\left[ mgy - \frac{1}{m}(p_x^2+p_y^2) - \kappa\left( xp_x+yp_y\right) \right] - mg.
		\end{split}
	\end{equation}

	A valid initial condition, $z_0 = (x(0),y(0),p_x(0),p_y(0))$, for \eqref{eq:pendulum_stable} needs to satisfy the constraint $\Phi(z_0)=0$. However, due to running numerical errors, the constraint will not be preserved as the state is integrated. By choosing $\kappa>0$, a correcting feedback is introduced to help preserve $\Phi=0$ as shown in Figure \ref{fig:stable_pendulum}.
\begin{figure}
	\begin{subfigure}{.45\textwidth}
		\centering
		\includegraphics[width=.9\linewidth]{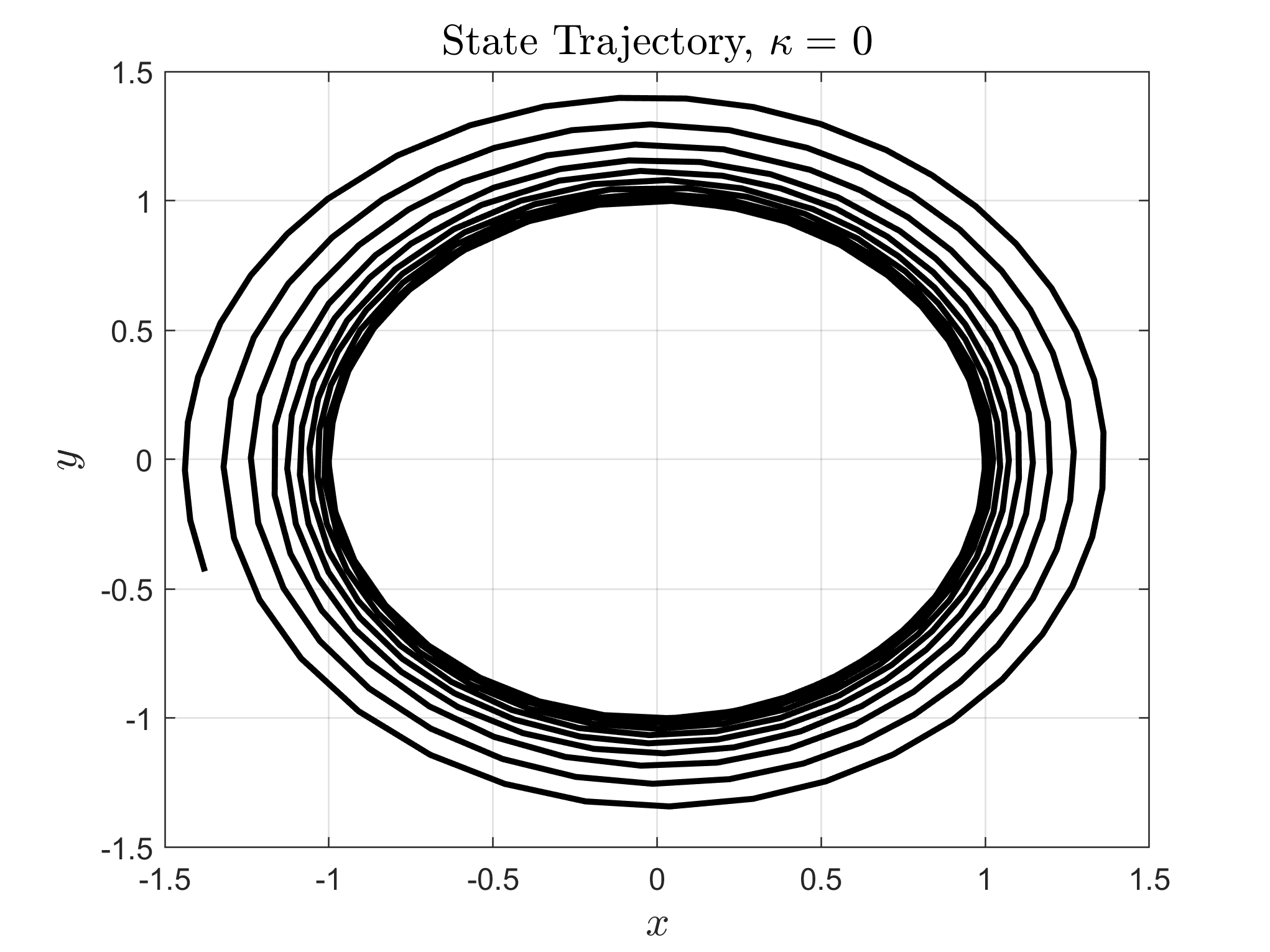}  
		\caption{State trajectory with $\kappa=0$.}
		\label{fig:sub-first}
	\end{subfigure}
	\begin{subfigure}{.45\textwidth}
		\centering
		\includegraphics[width=.9\linewidth]{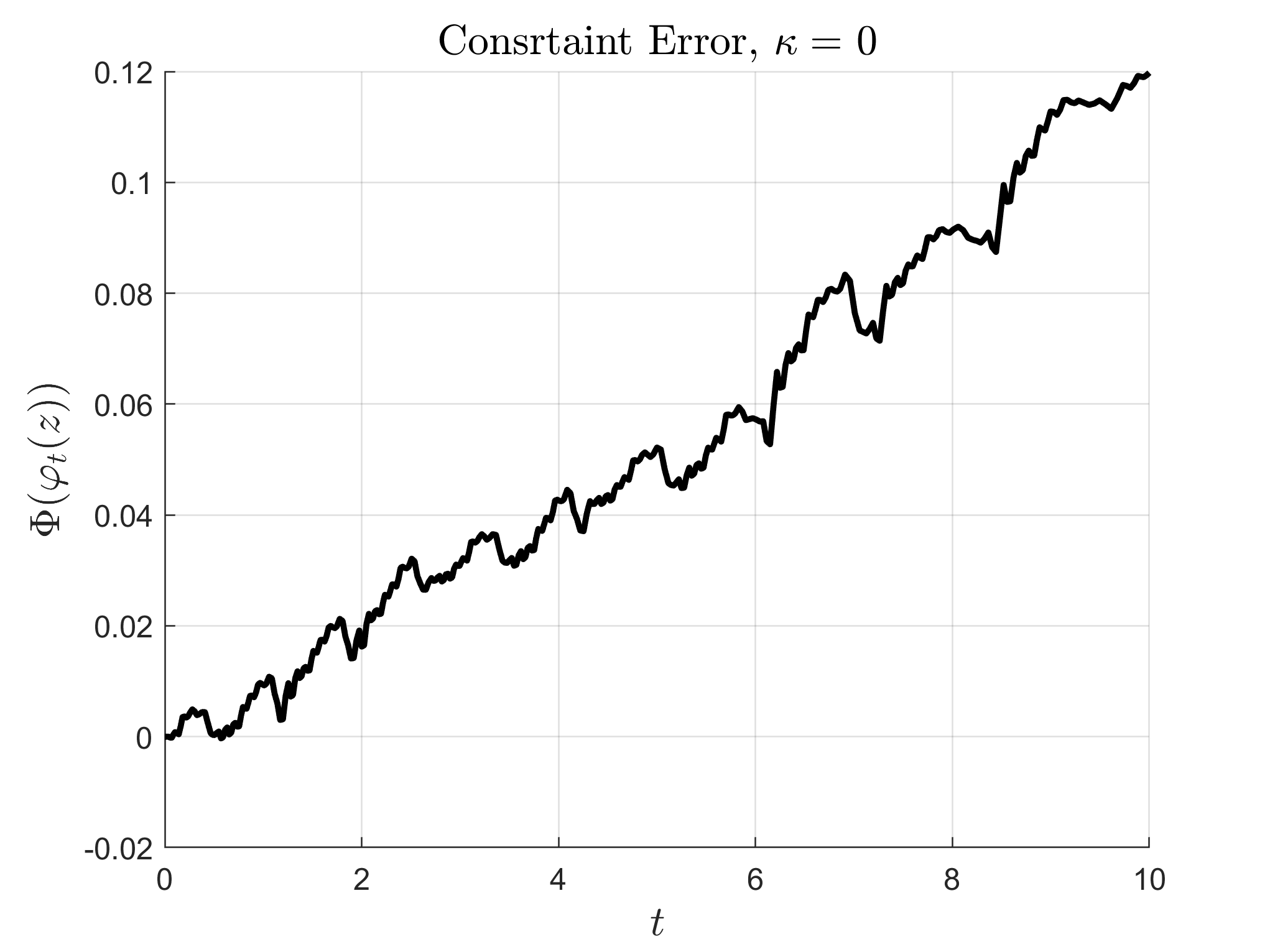}  
		\caption{Constraint trajectory with $\kappa=0$.}
		\label{fig:sub-second}
	\end{subfigure}
	\newline
	\begin{subfigure}{.45\textwidth}
		\centering
		\includegraphics[width=.9\linewidth]{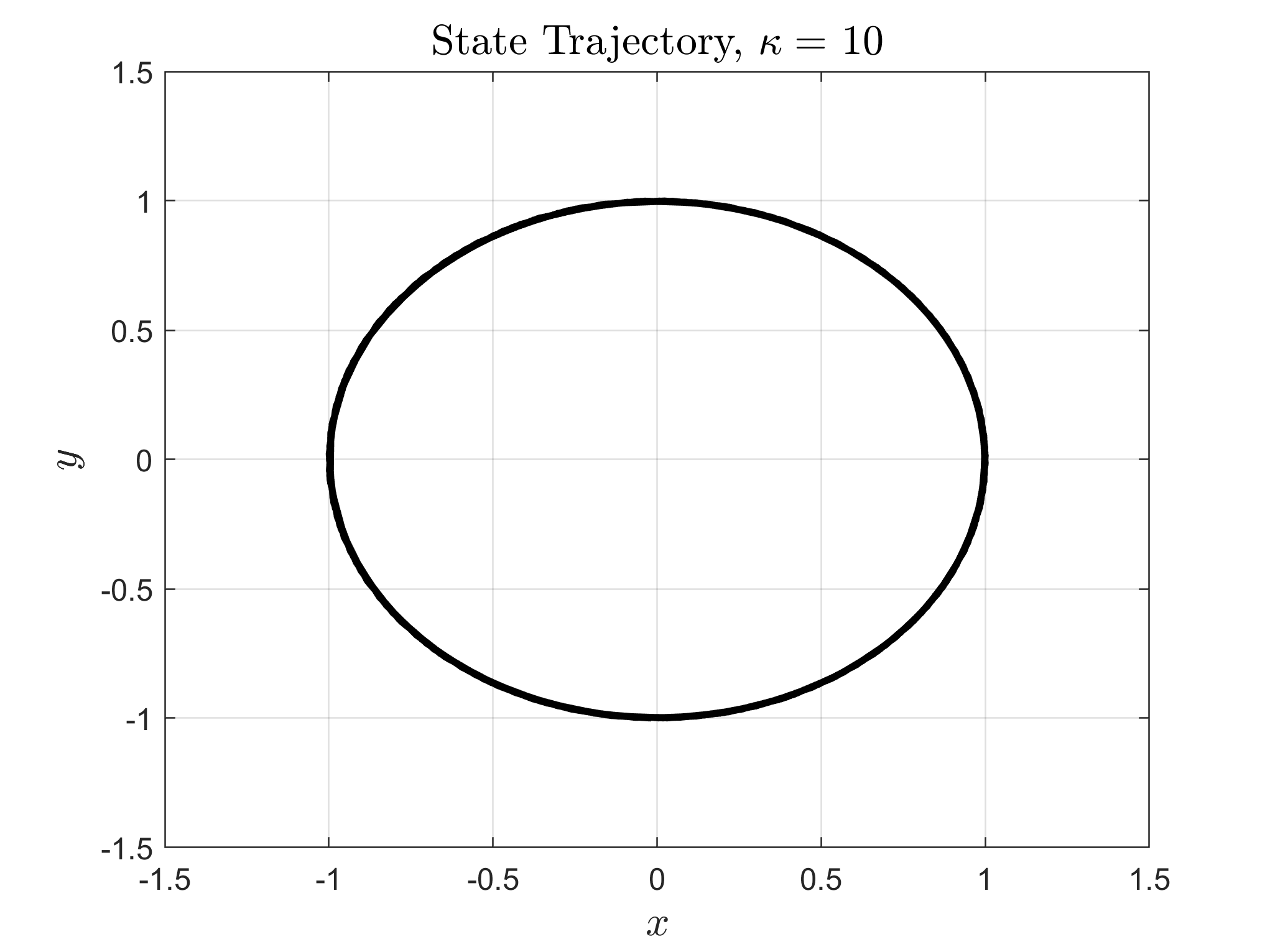}  
		\caption{State trajectory with $\kappa=10$.}
		\label{fig:sub-third}
	\end{subfigure}
	\begin{subfigure}{.45\textwidth}
		\centering
		\includegraphics[width=.9\linewidth]{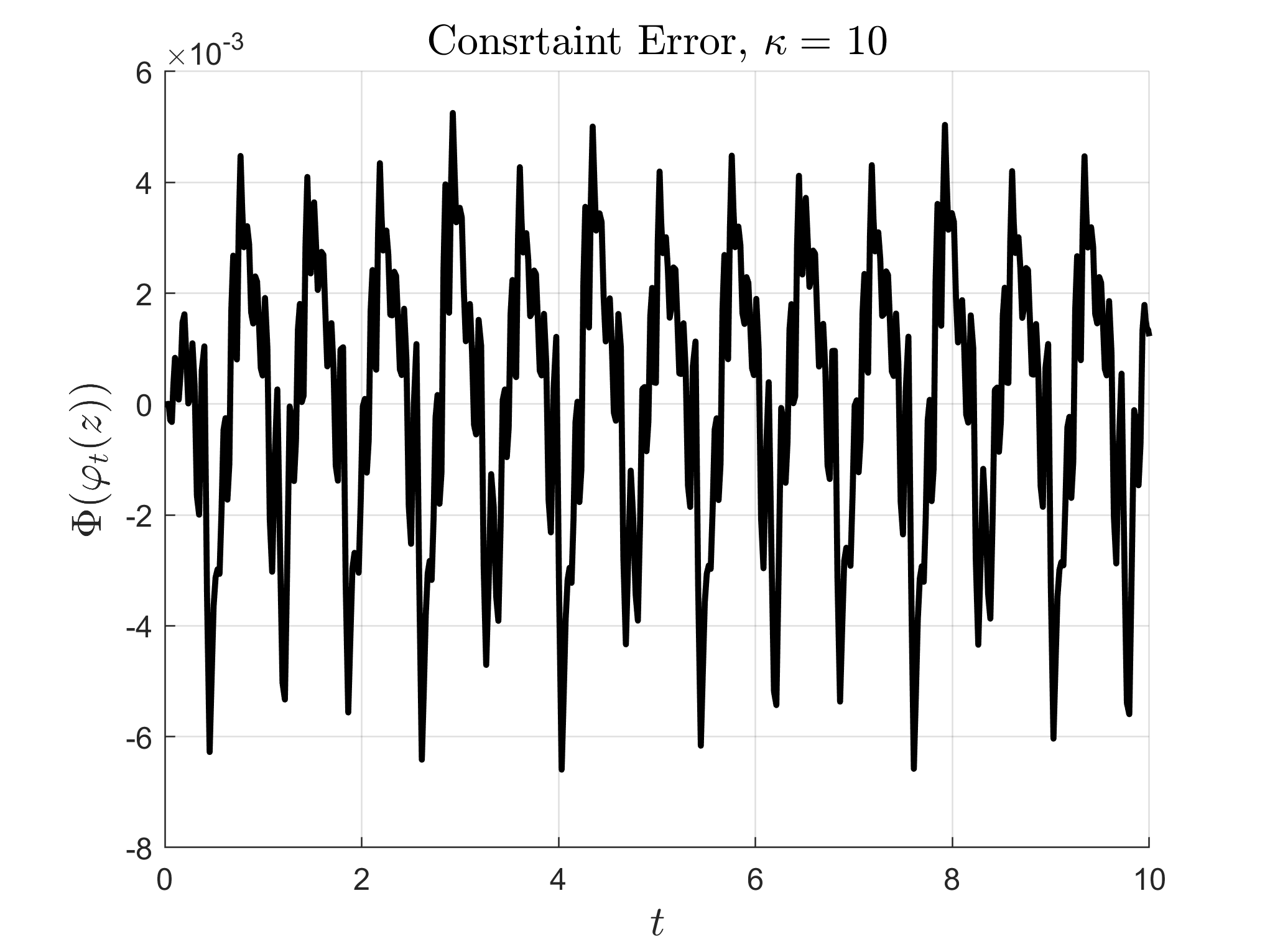}  
		\caption{Constraint trajectory with $\kappa=10$.}
		\label{fig:sub-fourth}
	\end{subfigure}
	\caption{The introduction of a positive $\kappa$ helps to stabilize the numerical trajectory. Numerical integration was performed with Matlab's \texttt{ode45}.}
	\label{fig:stable_pendulum}
\end{figure}
\end{example}
\subsection{Ideal Constraints}
We end this section with a brief consideration of ideal constraints - constraints that perform no work on the system. As discussed in \cite{deLeon1997}, a constraint $\Psi$ on $TQ$ is ideal if and only if
\begin{equation*}
	\left.\dot{q}^i\frac{\partial \Psi}{\partial \dot{q}^i}\right|_M = 0.
\end{equation*}
On the cotangent side, this condition becomes
\begin{equation*}
	\left.\mathcal{C}^*d\Phi\left( X_H\right)\right|_M = 0.
\end{equation*}
In particular, if all the constraints are ideal, then energy is conserved. Indeed,
\begin{equation*}
	\dot{H} = m_{\alpha\beta}\{H,\Phi^\alpha\}\mathcal{C}^*d\Phi^\beta(X_H).
\end{equation*}
Constraints will not be assumed to be ideal. Indeed, examples \ref{ex:rotating_disk}, \ref{ex:rotating_sphere}, and \ref{ex:constant_KE} will all be subject to non-ideal constraints. In particular, these examples will be volume-preserving but not energy-preserving (although a modified energy integral exists for the first two examples \cite{fgn2018}).

\section{Nonholonomic Volume}\label{sec:volume}

For a given nonholonomic vector field $X_H^M\in\mathfrak{X}(M)$, a volume-form $\mu\in \Omega^{2n-k}(M)$ is invariant if and only if
\begin{equation*}
	\mathrm{div}_\mu(X_H^M) = 0.
\end{equation*}
This section offers a way to choose a volume-form $\mu$ based off a (non-canonical choice) of realization $\mathscr{C}$ as well as showing how to compute the divergence of the nonholonomic vector fields with respect to this volume.
\subsection{Nonholonomic Volume form}
The symplectic manifold $T^*Q$ has a canonical volume form $\omega^n$. However, the nonholonomic flow takes place on a submanifold $M\subset T^*Q$ which is $2n-k$ dimensional. Therefore, $\omega^n$ is \textit{not} a volume form on $M$. Here, we construct a volume form on $M$ which is unique up to the choice of realization. The derivation of this will be similar to the construction of the volume form on an energy surface in \S 3.4 of \cite{abraham2008foundations}.
For the realization $\mathscr{C} = \{\Phi^1,\ldots,\Phi^k\}$, define the $k$-form
\begin{equation*}
\sigma_\mathscr{C} := d\Phi^1\wedge\ldots\wedge d\Phi^k.
\end{equation*}
\begin{definition}\label{def:nh_volume_form}
	If we denote the inclusion map by $\iota:M\hookrightarrow T^*Q$, then a nonholonomic volume, {$\mu_\mathscr{C}$}, is given by
	\begin{equation*}
	\mu_\mathscr{C} = \iota^*\varepsilon,\quad \sigma_\mathscr{C}\wedge\varepsilon = \omega^n.
	\end{equation*}
\end{definition}

\begin{proposition}
	Given an ordered collection of constraints, $\mathscr{C}$, the induced volume form $\mu_\mathscr{C}$ is unique.
\end{proposition}
\begin{proof}
	Suppose that $\varepsilon$ and $\varepsilon'$ are two forms satisfying $\sigma_\mathscr{C}\wedge\varepsilon = \omega^n$. Then
	\begin{equation*}
	\varepsilon - \varepsilon' = \alpha,\quad \sigma_{\mathscr{C}}\wedge\alpha = 0.
	\end{equation*}
	Now let $\iota:M\hookrightarrow T^*Q$ be the inclusion. Then from the above, we see that
	\begin{equation*}
	\iota^*\varepsilon = \iota^*\varepsilon' + \iota^*\alpha.
	\end{equation*}
	The result will follow so long as $\iota^*\alpha = 0$. Suppose that $\iota^*\alpha\ne0$ and choose vectors $v^1,\ldots,v^{2n-k}\in T_xM\subset T_xT^*Q$ such that $\alpha(v^1,\ldots,v^{2n-k})\ne 0$. Complete this collection of vectors to a basis of $T_xT^*Q$: $v^1,\ldots,v^{2n-k},v^{2n-k+1},\ldots,v^{2n}$ such that $\sigma_{\mathscr{C}}(v^{2n-k+1},\ldots,v^{2n})\ne 0$. Then we have
	\begin{equation*}
	\sigma_{\mathscr{C}}\wedge \alpha \left(v^1,\ldots,v^{2n}\right) = 
	(-1)^{(2n-k)k}\alpha(v^1,\ldots,v^{2n-k})\cdot\sigma_{\mathscr{C}}(v^{2n-k+1},\ldots,v^{2n}) \ne 0,
	\end{equation*}
	which is a contradiction.
\end{proof}

\begin{remark}
	Notice that for an ordered collection of constraints the volume form is unique. However, changing the order of the constraints changes the sign of the induced volume form and re-scaling constraints re-scales the volume form. In this sense, $\mathscr{C}$ uniquely determines $\mu_\mathscr{C}$, but $M$ only determines $\mu_\mathscr{C}$ up to a multiple. This is to be expected as all volume-forms are related by a multiplicative factor.
\end{remark}

While examining the failure of Liouville's theorem (Theorem \ref{th:liouville}) for nonholonomic systems, we will see when $\mu_\mathscr{C}$ is preserved under the flow of $X_H^M$. More generally, we will consider the existence of a smooth density $f\in C^\infty(M)$ when $f\mu_\mathscr{C}$ is preserved.
\subsection{Divergence}
Let $\omega = dq^i\wedge dp_i$ be the standard symplectic form on $T^*Q$. This in turn induces a volume form $\omega^n$. A measure of how much a flow fails to preserve a volume form is described by its divergence. 
We proceed with computing the divergence of a nonholonomic vector field, $\mathrm{div}_{\mu_{\mathscr{C}}}(X_H^M)$. When this is nonzero, we will be interested in finding a density, $f$, such that $\mathrm{div}_{f\mu_{\mathscr{C}}}(X_H^M)=0$. This problem will be addressed in \S\ref{sec:cohom}.

Before we begin with the divergence calculation, we first present a helpful lemma which allows us to relate the divergence of the extended nonholonomic vector field with the corresponding restricted vector field.
\begin{lemma}\label{le:subvol}
	If $\mathscr{C}$ is a realization of the constraint manifold $M\subset T^*Q$, then
	\begin{equation*}
	\left.\mathrm{div}_{\omega^n}\left(\Xi_H^{\mathscr{C}}\right)\right|_{M} = \mathrm{div}_{\mu_\mathscr{C}} \left(X_H^M\right).
	\end{equation*}
\end{lemma}
\begin{proof}
	Leibniz's rule for the Lie derivative provides
	\begin{equation*}
	\begin{split}
	\mathcal{L}_{\Xi_H^{\mathscr{C}}}\omega^n &= \mathcal{L}_{\Xi_H^{\mathscr{C}}}\left( \sigma_\mathscr{C}\wedge\varepsilon\right) \\
	&= \left(\mathcal{L}_{\Xi_H^{\mathscr{C}}}\sigma_\mathscr{C}\right)\wedge\varepsilon + \sigma_\mathscr{C}\wedge\left(\mathcal{L}_{\Xi_H^{\mathscr{C}}}\varepsilon\right).
	\end{split}
	\end{equation*}
	However, $\mathcal{L}_{\Xi_H^{\mathscr{C}}}\sigma_\mathscr{C} = 0$ because the constraints are preserved under the flow.
	Applying this, we see that
	\begin{equation*}
	\mathcal{L}_{\Xi_H^{\mathscr{C}}}\omega^n = 
	\sigma_\mathscr{C}\wedge \left(\mathcal{L}_{\Xi_H^{\mathscr{C}}}\varepsilon\right),
	\end{equation*}
	which gives
	\begin{equation*}
	\left(\mathrm{div}_{\omega^n}\left(\Xi_H^\mathscr{C}\right)\right) \sigma_{\mathscr{C}}\wedge\varepsilon = \sigma_{\mathscr{C}}\wedge\left( \mathcal{L}_{\Xi_H^\mathscr{C}}\varepsilon\right).
	\end{equation*}
	Due to the fact that the Lie derivative commutes with restriction, the result follows.
\end{proof}
This lemma allows for us to calculate the divergence of the extended nonholonomic vector field and to restrict to the constraint distribution \textit{afterwards}.

Using Cartan's magic formula along with Lemma \ref{le:subvol}, we see
\begin{equation*}
	\begin{split}
		\mathcal{L}_{\Xi_H^\mathscr{C}}\left( \omega^n\right) &=
		i_{\Xi_H^\mathscr{C}}d\omega^n + di_{\Xi_H^\mathscr{C}}\omega^n \\
		&= n\cdot d\left( i_{\Xi_H^\mathscr{C}}\omega\wedge\omega^{n-1}\right) \\
		&= n\cdot d\left( i_{\Xi_H^\mathscr{C}}\omega\right)\wedge\omega^{n-1} - n\cdot\left( i_{\Xi_H^\mathscr{C}}\omega\right)\wedge d\omega^{n-1} \\
		&= n\cdot d\nu_H^\mathscr{C}\wedge\omega^{n-1}.
	\end{split}
\end{equation*}
The divergence of the system is controlled by the failure of $\nu_H^\mathscr{C}$ to be closed. Its derivative is given by
\begin{equation*}
	d\nu_H^\mathscr{C} = ddH - d\left(m_{\alpha\beta}\left\{ H, \Phi^\alpha\right\} \cdot \mathcal{C}^*d\Phi^\beta\right).
\end{equation*}
Let $\phi_\alpha = m_{\alpha\beta}\Phi^\beta$. As $\Phi^\beta=0$ on $M$, the derivative becomes (where restriction to $M$ is implied)
\begin{equation*}
	d\nu_H^\mathscr{C} = -d\{H,\phi_\beta\}\wedge \mathcal{C}^*d\Phi^\beta + \{H,\phi_\beta\}\cdot d\mathcal{C}^*d\Phi^\beta.
\end{equation*}
The form $d\nu_H^\mathscr{C}$ is a general 2-form. However, the only terms that survive once it is wedged with $\omega^{n-1}$ are the diagonal terms, $dq^i\wedge dp_i$. Using local coordinates, $a_i^\beta dq^i = \mathcal{C}^*d\Phi^\beta$, we have
\begin{equation*}
	\left( d\nu_H^\mathscr{C}\right)_{\mathrm{diag}} =
	\left(\frac{\partial\{H,\phi_\beta\}}{\partial p_i}a_i^\beta - \{ H,\phi_\beta\} \frac{\partial a_i^\beta}{\partial p_i} \right) dq^i\wedge dp_i.
\end{equation*}
Therefore, the divergence is given by
\begin{equation*}
	\mathrm{div}_{\mu_{\mathscr{C}}}(X_H^M) = n\cdot \left(\frac{\partial\{H,\phi_\beta\}}{\partial p_i}a_i^\beta - \{ H,\phi_\beta\} \frac{\partial a_i^\beta}{\partial p_i} \right).
\end{equation*}
The first component of this expression can be written intrinsically. Notice that
\begin{equation*}
	d\pi_Q\cdot X_f = \frac{\partial f}{\partial p_i}\frac{\partial}{\partial q^i},
\end{equation*}
where $\pi_Q:T^*Q\to Q$ is the cotangent bundle projection. Therefore, the divergence can now be written as
\begin{equation}\label{eq:general_divergence}
	\mathrm{div}_{\mu_{\mathscr{C}}}(X_H^M) = -n\cdot \left(\mathcal{C}^*d\Phi^\beta\left( [X_H,X_{\phi_\beta}]\right) + \{ H,\phi_\beta\} \frac{\partial a_i^\beta}{\partial p_i} \right).
\end{equation}
\subsection{Intrinsic Forms of the Divergence}
Suppose the Lagrangian is natural with Riemannian metric $g = (g_{ij})$. Then
\begin{equation*}
	\mathcal{C}^*d\Phi^\beta = g_{ij}\frac{\partial \Phi^\beta}{\partial p_j} dx^i, \quad a_i^\beta = g_{ij}\frac{\partial \Phi^\beta}{\partial p_j}, \quad \frac{\partial a_i^\beta}{\partial p_i} = g_{ij}\frac{\partial^2 \Phi^\beta}{\partial p_i\partial p_j}.
\end{equation*}
Let us call this final double sum $\mathcal{M}^\beta$. Applying this to \eqref{eq:general_divergence}, we have
\begin{equation}\label{eq:natural_divergence}
	\mathrm{div}_{\mu_{\mathscr{C}}}(X_H^M) = -n\cdot \mathcal{C}^*d\Phi^\beta\left( [X_H,X_{\phi_\beta}] \right) - n\cdot \left\{ H,\phi_\beta\right\}\mathcal{M}^\beta.
\end{equation}
Next, assume that (in addition to the Lagrangian being natural), the constraints are affine which makes $\mathcal{M}^\beta=0$. Also the matrix, $m^{\alpha\beta}$, does not depend on $p$ and therefore, the divergence can be written as
\begin{equation}\label{eq:affine_divergence}
	\mathrm{div}_{\mu_{\mathscr{C}}}(X_H^M) = -n\cdot m_{\alpha\beta}\cdot \pi_Q^*\eta^\beta\left( [X_H, X_{\Phi^\alpha}] \right),
\end{equation}
as $\mathcal{C}^*d\Phi^\beta = \pi_Q^*\eta^\beta$ for affine constraints.

For affine systems, the divergence can be expressed intrinsically via \eqref{eq:affine_divergence}. To understand the divergence in the general nonlinear case, we need an intrinsic way to interpret the term $\mathcal{M}^\beta$. This can be accomplished via the following procedure:
\begin{equation*}
	\begin{tikzcd}[column sep=large]
		C^\infty(T^*Q) \arrow[r, "\mathcal{C}^*d"] & \Omega^1(T^*Q) \arrow[r, "-\omega^\sharp"] & \mathfrak{X}(T^*Q) \arrow[r, "\mathrm{div}_{\omega^n}(\cdot)"] & C^\infty(T^*Q) \\[-2em]
		\Phi^\beta \arrow[r, mapsto] \arrow[rrr, bend right=25, mapsto, "\Delta_\mathcal{C}"] & a_i^\beta dx^i \arrow[r, mapsto] & \displaystyle a_i^\beta \frac{\partial}{\partial p_i} \arrow[r, mapsto] & \displaystyle \frac{\partial a_i^\beta}{\partial p_i}
	\end{tikzcd}
\end{equation*}
The divergence for a nonholonomic system subject to nonlinear constraints is thus
\begin{equation}
	\mathrm{div}_{\mu_\mathscr{C}}(X_H^M) = \dim(Q)\cdot \left( \mathcal{C}^*d\Phi^\beta\left( \left[ X_H, X_{\phi_\beta}\right] \right) + \left\{ H, \phi_\beta\right\} \Delta_\mathcal{C}\left(\Phi^\beta\right) \right),
\end{equation}
which is the key to proving the main result, Theorem \ref{th:main_proved}.

\section{Invariant Volumes and the Cohomology Equation}\label{sec:cohom}
In general, the divergence of a nonholonomic system does not vanish as \eqref{eq:general_divergence} shows. When does there exist a \textit{different} volume form on $M$ that is invariant under the flow? i.e. does there exist a density $\rho>0$ such that $\mathrm{div}_{\rho\mu_{\mathscr{C}}}(X_H^M)=0$? Finding such a $\rho$ requires solving a certain type of partial differential equation which is known as the smooth dynamical cohomology equation. Solving this PDE is generally quite difficult. However, when the constraints are affine and the density is assumed to be basic, the problem becomes considerably more tractable 
\subsection{The Cohomology Equation}
What conditions need to be met for $\rho$ such that $\rho\mu_\mathscr{C}$ is an invariant volume form? Using the formula for the divergence as well as the fact that the Lie derivative is a derivation yields:
\begin{equation*}
\mathrm{div}_{\rho\mu_\mathscr{C}}\left(X_H^M\right) = \mathrm{div}_{\mu_\mathscr{C}}\left(X_H^M\right) + \frac{1}{\rho}\mathcal{L}_{X_H^M}(\rho).
\end{equation*}
Therefore the density, $\rho$, yields an invariant measure if and only if
\begin{equation}\label{eq:lie_deriv}
\frac{1}{\rho}\mathcal{L}_{X_H^M}(\rho) = -\mathrm{div}_{\mu_\mathscr{C}}\left( X_H^M\right).
\end{equation}
Notice that the left hand side of \eqref{eq:lie_deriv} can be integrated to
\begin{equation*}
\frac{1}{\rho}\mathcal{L}_{X_H^M}(\rho) = d\left(\ln \rho \right) \left( X_H^M\right).
\end{equation*}
Calling $\lambda = \ln \rho$, we have the following lemma.
\begin{lemma}
	For a nonholonomic vector field, $X_H^M$, there exists a smooth invariant volume, $\rho\mu_\mathscr{C}$, if there exists an exact 1-form $\alpha = d\lambda$ such that 
	\begin{equation}\label{eq:measpde}
	\begin{split}
	\alpha\left( X_H^M\right) &= -\mathrm{div}_{\mu_\mathscr{C}}\left(X_H^M\right).
	\end{split}
	\end{equation}
	Then the density is (up to a multiplicative constant) $\rho = e^\lambda$.
\end{lemma}
Therefore the existence of invariant volumes boils down to finding global solutions to the PDE \eqref{eq:measpde}. Using the divergence calculation, \eqref{eq:general_divergence}, we can state the main result.
\begin{theorem}\label{th:main_proved}
	The nonholonomic Hamiltonian vector field $X_H^M$ possesses a smooth invariant volume, $\rho\mu_{\mathscr{C}}$, if there exists an exact 1-form $\alpha\in\Omega^1(T^*Q)$ such that
	\begin{equation*}
		-\alpha\left( X_H^M\right) = \dim(Q)\cdot \left(\mathcal{C}^*d\Phi^\beta\left( [X_H,X_{\phi_\beta}]\right) + \{ H,\phi_\beta\} \Delta_\mathcal{C}\left(\Phi^\beta\right) \right).
	\end{equation*}
\end{theorem}

The remainder of this section deals with uniqueness of solutions and a necessary condition for solutions to exist.
\begin{remark}
	PDEs of the form $dg(X) = f$ for a given smooth function $f$, vector field $X$ and with $g$ as the unknown are called cohomology equations \cite{forni1995cohomological, livvsic1972cohomology}. Thus the equation \eqref{eq:measpde} is a cohomology equation.
\end{remark}
\subsubsection{Uniqueness}
The problem of existence is quite difficult in general and we postpone that discussion until the next subsection where we assume that the solution has the form $\rho=\pi_Q^*\tilde{\rho}$. In the meantime, assuming that there exists a function $\lambda\in C^\infty(M)$ that solves \eqref{eq:measpde}, do there exist other solutions? Suppose that $\lambda_1$ and $\lambda_2$ both solve \eqref{eq:measpde}. Then their difference must be a first integral of the system: $\mathcal{L}_{X_H^M}(\lambda_1-\lambda_2) = 0$. Solutions of \eqref{eq:measpde} are then unique up to constants of motion. i.e. if $\lambda$ solves \eqref{eq:measpde}, then every invariant density has the form (again, up to a multiplicative constant)
\begin{equation*}
\rho = \exp\left( \lambda + \mathrm{constant~of~motion} \right).
\end{equation*}
Therefore invariant measures can be thought of as an affine space with dimension being equal to the number of first integrals of the nonholonomic system.

\subsection{Special Case: Basic Densities and Affine Constraints}\label{sec:configuration}
In general, solving the cohomology equation \eqref{eq:measpde} is quite difficult. It turns out, however, that it is \textit{relatively} easy to determine necessary and sufficient conditions on the solvability when the density is assumed to be basic and the constraints are affine.

For the remainder of this section, unless otherwise stated, the constraints are assumed to be affine. In particular, if $N\subset TQ$, then the constraints have the form
\begin{equation*}
	\Psi^\alpha(v) = \eta^\alpha(v) + \tau_Q^*\xi^\alpha(v),
\end{equation*}
for 1-forms $\eta^\alpha$ and function $\xi^\alpha$. On the cotangent side, the constraints become
\begin{equation}\label{eq:affine_constraint}
	\Phi^\alpha(p) = P(W^\alpha)(p) + \pi_Q^*\xi^\alpha(p),
\end{equation}
where $W^\alpha = \mathbb{F}L^{-1}\eta^\alpha$.

\begin{definition}
	A density $\rho:T^*Q\to\mathbb{R}$ is said to be basic if $\rho=\pi_Q^*\tilde{\rho}$ for some $\tilde{\rho}:Q\to\mathbb{R}$.
\end{definition}

Under this assumption, \eqref{eq:affine_divergence} can be presented in a surprisingly nice way. In this case, the divergence can be described by an equivalence class of ``affine-forms.'' The density form and density function, defined below, is a representative element from this class.
\begin{definition}
	Let $\mathscr{C}$ be an affine realization of $M\subset T^*Q$ of the form \eqref{eq:affine_constraint}. Then, define the \textit{density form} and the \textit{density function} to be
	\begin{equation*}
		\begin{split}
			\vartheta_{\mathscr{C}} &= m_{\alpha\beta}\cdot\mathcal{L}_{W^\alpha}\eta^\beta, \\
			\zeta_\mathscr{C} &= m_{\alpha\beta}\cdot\mathcal{L}_{W^\alpha}\xi^\beta,
		\end{split}
	\end{equation*}
	respectively.
\end{definition}
Studying this pair, $(\vartheta_{\mathscr{C}},\zeta_\mathscr{C})$, provides necessary and sufficient conditions for the existence of basic densities. 
\begin{theorem}\label{th:affine_basic}
	For a natural Hamiltonian system subject to affine constraints, there exists an invariant volume of the form $(\pi_Q^*\tilde{\rho})\mu_{\mathscr{C}}$ if and only if there exists functions $\varphi_\gamma$ such that
	\begin{equation}\label{eq:affine_volume}
		\begin{split}
			n\cdot \vartheta_{\mathscr{C}} - \varphi_\gamma\eta^\gamma &= -d\ln\tilde{\rho}, \\
			n\cdot \zeta_\mathscr{C} - \varphi_\gamma\xi^\gamma &= 0.
		\end{split}
	\end{equation}
\end{theorem}
\begin{proof}
	To prove this, we will show that $-n\cdot\pi_Q^*\vartheta(X_H^M) - n\cdot \pi_Q^*\zeta_\mathscr{C} = \mathrm{div}_\mathscr{C}(X_H^M)$. Recall that the differential of a 1-form is given by $d\alpha(X,Y) = X\alpha(Y)-Y\alpha(X) - \alpha([X,Y])$ and that $\pi_Q^*\eta^\beta(X_H)|_M = -\xi^\beta$. Returning to \eqref{eq:affine_divergence}, we have
	\begin{equation*}
		\begin{split}
			\mathrm{div}_{\mu_\mathscr{C}}(X_H^M) &= -n\cdot m_{\alpha\beta}\cdot \pi_Q^*\eta^\beta\left( [X_H, X_{\Phi^\alpha}] \right) \\
			&= -n\cdot m_{\alpha\beta} \cdot \left(
			X_Hm^{\alpha\beta} + X_{\Phi^\alpha}\xi^\beta - \pi_Q^*d\eta^\alpha(X_H,X_{\Phi^\alpha}) \right) \\
			&= -n\cdot m_{\alpha\beta} \cdot \left( dm^{\alpha\beta}(\dot{q}) + d\xi^\beta(W^\alpha) - d\eta^\alpha(\dot{q},W^\beta) \right) \\
			&= -n\cdot m_{\alpha\beta}\cdot\left[ \left( di_{W^\beta}\eta^\alpha + i_{W^\beta}d\eta^\alpha\right)(\dot{q}) + d\xi^\beta(W^\alpha) \right] \\
			&= -n\cdot m_{\alpha\beta}\cdot \left[ \mathcal{L}_{W^\beta}\eta^\alpha(\dot{q}) + \mathcal{L}_{W^\alpha}\xi^\beta \right] \\
			&= -n\cdot\vartheta_{\mathscr{C}}(\dot{q}) - n\cdot\zeta_\mathscr{C}.
		\end{split}
	\end{equation*}
	This computation shows that $\mathrm{div}_{\mu_{\mathscr{C}}}(X_H^M) = -n\cdot\vartheta_{\mathscr{C}}(\dot{q})-n\cdot\zeta_\mathscr{C}$, but $\dot{q}$ cannot be arbitrary as it must lie within $M$ which states that $\eta^\alpha(\dot{q}) + \xi^\alpha = 0$. Adding multiples of the constraints to the divergence to produce an exact 1-form yields \eqref{eq:affine_volume}.
\end{proof}

This theorem allows for a straight-forward algorithm to find invariant volumes in nonholonomic systems subject to affine constraints; one only needs to compute the pair $(\vartheta_{\mathscr{C}},\zeta_\mathscr{C})$ and determine whether or not it can be made exact by appending constraints to it. This procedure will be carried out on multiple examples in \S\ref{sec:examples}.
\begin{remark}
	In the pure kinetic energy case with \textit{linear constraints} discussed in \cite{federovnaranjo}, it is proved that if the system admits an (arbitrary) invariant volume, then one can always find another invariant volume form whose density function depends only on the (reduced) configuration variables. Moreover, systems subjected to affine constraints may possess an invariant volume whose density is not basic, cf. \cite{nonbasicaffine}.
\end{remark}

The above shows that ``exactness'' of $(\vartheta_{\mathscr{C}},\zeta_\mathscr{C})$ determines the existence of a density depending on configuration. How does this depend on the choice of $\mathscr{C}$ to realize the constraints? It turns out the answer is independent of the choice of realization.
\begin{theorem}\label{th:density_realization}
	Let $\mathscr{C}_1$ and $\mathscr{C}_2$ both be affine realizations of the constraint manifold $M$. Suppose that there exist functions $\varphi_\gamma$ such that
	\begin{equation*}
		\vartheta_{\mathscr{C}_1} - \varphi_\gamma\eta^\gamma_1 = -d\ln\tilde{\rho}_1, \quad \zeta_{\mathscr{C}_1} - \varphi_\gamma\xi^\gamma_1 = 0,
	\end{equation*}
	then there exists functions $\psi_\gamma$ such that
	\begin{equation*}
		\vartheta_{\mathscr{C}_2} - \psi_\gamma\eta^\gamma_2 = -d\ln\tilde{\rho}_2, \quad \zeta_{\mathscr{C}_2} - \psi_\gamma\xi^\gamma_2 = 0.
	\end{equation*}
	Moreover, $(\pi_Q^*\tilde{\rho}_1)\mu_{\mathscr{C}_1} = (\pi_Q^*\tilde{\rho}_2)\mu_{\mathscr{C}_2}$ modulo a constant of motion.
\end{theorem}
\begin{proof}
	Let $\eta_2^\gamma = c_\alpha^\gamma\eta_1^\alpha$ and $\xi_2^\gamma = c_\alpha^\gamma \xi_1^\alpha$ where $c_\alpha^\gamma$ is the required coordinate change. For ease of notation, we will drop the subscript ``1.'' 
	The matrices are related via
	\begin{equation*}
		m_2^{\alpha\beta} = c_\gamma^\alpha c_\delta^\beta m_1^{\gamma\delta}.
	\end{equation*}
	The density form is
	\begin{equation*}
		\begin{split}
			m_2^{\alpha\beta}\vartheta_2 &= \mathcal{L}_{c_\gamma^\alpha W^\gamma} \left(c_\delta^\beta\eta^\delta\right) \\
			&= c_\gamma^\alpha c_\delta^\beta \mathcal{L}_{W^\gamma}\eta^\delta + c_\delta^\beta m_1^{\delta\gamma}dc_\gamma^\alpha + c_\gamma^\alpha dc_\delta^\beta(W^\gamma)\eta^\delta \\
			&= m_2^{\alpha\beta}\vartheta_1 + m_2^{\alpha\beta}\sigma_\alpha^\gamma dc_\gamma^\alpha + c_\gamma^\alpha dc_\delta^\beta(W^\gamma)\eta^\delta,
		\end{split}
	\end{equation*}
	where $(\sigma_\alpha^\gamma) := (c_\gamma^\alpha)^{-1}$. By Lemma \ref{le:barvinok} below, we have
	\begin{equation}
		m_2^{\alpha\beta}\vartheta_2 = m_2^{\alpha\beta}\vartheta_1 + m_2^{\alpha\beta} d\left[ \ln \det(c_\gamma^\alpha)\right] + c_\gamma^\alpha dc_\delta^\beta(W^\gamma)\eta^\delta.
	\end{equation}
	This shows that $\vartheta_1$ and $\vartheta_2$ differ by something exact and a multiple of the constraining 1-forms. 
	Using the fact that $\vartheta_1 - \varphi_\gamma\eta_1^\gamma = -d \ln\tilde{\rho}_1$, we see that
	\begin{equation*}
		\vartheta_2 - \left( \varphi_\gamma + C_\gamma\right) \eta^\gamma = d\left[ -\ln\tilde\rho_1 + \ln\det(c_\gamma^\alpha) \right],
	\end{equation*}
	where $m_2^{\alpha\beta}C_\gamma = c_\delta^\alpha dc_\gamma^\beta(W^\delta)$. This provides
	\begin{equation*}
		\tilde{\rho}_2 = \det(\sigma_\alpha^\gamma)\tilde{\rho}_1, \quad
		\psi_\gamma = c_\gamma^\delta \left[ \varphi_\delta + C_\delta\right].
	\end{equation*}
	These values of $\psi_\gamma$ make $\vartheta_2$ exact. We next check that these values of $\psi_\gamma$ make $\zeta_2$ vanish.
	\begin{equation*}
		\begin{split}
			m_2^{\alpha\beta}\left(\zeta_2 - \psi_\gamma c_\delta^\gamma \xi^\delta\right) &= 
			c_\gamma^\alpha c_\delta^\beta \mathcal{L}_{W^\gamma}\xi^\delta + c_\gamma^\alpha dc_\delta^\beta(W^\gamma)\xi^\delta - m_2^{\alpha\beta}\psi_\gamma c_\delta^\gamma \xi^\delta \\
			&= m_2^{\alpha\beta}\zeta_1 - m_2^{\alpha\beta}\varphi_\delta\xi^\delta = 0.
		\end{split}
	\end{equation*}
	It remains to show that $(\pi_Q^*\tilde{\rho}_1)\mu_{\mathscr{C}_1} = (\pi_Q^*\tilde{\rho}_2)\mu_{\mathscr{C}_2}$. This is equivalent to $\mu_{\mathscr{C}_1} = (\pi_Q^*\det(c_\gamma^\alpha))\mu_{\mathscr{C}_2}$. Recalling Definition \ref{def:nh_volume_form}, we have $\sigma_1 = d\Phi^1\wedge\ldots d\Phi^k$ and
	\begin{equation*}
		\sigma_2 = \bigwedge_{i=1}^k \, d\left( c_\gamma^i \Phi^\gamma\right) = \det(c_\gamma^\alpha) \sigma_1 + \bigwedge_{i=1}^k \, \Phi^\gamma dc_\gamma^i.
	\end{equation*}
	Upon applying the constraints, the last term disappears and we obtain
	\begin{equation*}
		\sigma_1\wedge\varepsilon_1 = \sigma_2\wedge\varepsilon_2 = \det(c_\gamma^\alpha)\sigma_1\wedge\varepsilon_2.
	\end{equation*}
	This implies the desired result.
\end{proof}
\subsubsection{Invariant Volumes of Holonomic Systems}
A reason why studying $(\vartheta_\mathscr{C},\zeta_\mathscr{C})$ is insightful is that it immediately demonstrates why holonomic systems systems are measure-preserving, i.e. when $\eta^\alpha = df^\alpha$ and $\xi^\alpha=0$, \eqref{eq:affine_volume} can always be solved.
This short section idemonstrates how our general theory collapses to the known holonomic case.
We start with a helpful lemma.
\begin{lemma}[\footnote{We thank Dr. Alexander Barvinok for help with this proof.}]
	\label{le:barvinok}
	$m_{\alpha\beta}\cdot dm^{\alpha\beta} = d\left[ \ln \det \left(m^{\alpha\beta}\right)\right]$.
\end{lemma}
\begin{proof}
	It suffices to check along a curve in the manifold. Let $\gamma:I\to Q$ be a curve and let $A(t) = \left(m^{\alpha\beta}\right)\circ\gamma(t)$ be the mass matrix along the curve. Note that $A(t)$ is positive-definite and changes smoothly with $t$. We have
	\begin{equation*}
	\frac{d}{dt}\ln\det A(t) = \frac{\frac{d}{dt}\det A(t)}{\det A(t)} = \sum_{i=1}^m \, \frac{\det A_i(t)}{\det A(t)},
	\end{equation*}
	where $A_i(t)$ is obtained from $A(t)$ by differentiating the $i$-th row and leaving all other rows intact, i.e.
	\begin{equation*}
	A_i(t) = \left( \begin{array}{ccc}
	a_{11}(t) & \cdots & a_{1m}(t) \\
	\vdots & \ddots & \vdots \\
	a_{(i-1)1}(t) & \cdots & a_{(i-1)m}(t) \\
	a'_{i1}(t) & \cdots & a'_{im}(t) \\
	a_{(i+1)1}(t) & \cdots & a_{(i+1)m}(t) \\
	\vdots & \ddots & \vdots \\
	a_{m1}(t) & \cdots & a_{mm}(t)
	\end{array}\right).
	\end{equation*}
	Expanding $\det A_i(t)$ along the $i$-th row:
	\begin{equation*}
	\det A_i(t) = \sum_{j=1}^m \, (-1)^{i+j-1} a'_{ij}(t)\det A_{ij}(t),
	\end{equation*}
	where $A_{ij}(t)$ is the $(m-1)\times (m-1)$ matrix obtained from $A_i(t)$ and hence from $A(t)$ by crossing out the $i$-th row and $j$-th column. 
	
	Next, observe that $(-1)^{i+j-1}\det A_{ij}/\det A(t)$ is the $(j,i)$-th entry of the inverse matrix $A^{-1}(t)=(b_{ij})(t)$, and since $A(t)$ is symmetric, is also the $(i,j)$-th entry of $(b_{ij})(t)$. Summarizing,
	\begin{equation*}
	\frac{d}{dt}\ln\det A(t) = \sum_{i,j=1}^m \, a_{ij}'(t)b_{ij}(t).
	\end{equation*}
\end{proof}
\begin{proposition}
	If the constraints are holonomic, then there exist function $\varphi_\gamma$ such that $\vartheta_{\mathscr{C}}-\varphi_\gamma\eta^\gamma$ is exact. In particular, if $\mathscr{C}$ is chosen such that all $\eta^\alpha$ are closed, $\vartheta_{\mathscr{C}}$ is exact.
\end{proposition}
\begin{proof}
	When the constraints are holonomic, the 1-forms $\eta^\alpha$ can be chosen such that they are closed. Then the density form is 
	\begin{equation*}
	\begin{split}
	\vartheta_\mathscr{C} &= m_{\alpha\beta}\left( di_{W^\beta}\eta^\alpha + i_{W^\beta}\cancel{d\eta^\alpha}\right) \\
	&= m_{\alpha\beta}\cdot dm^{\alpha\beta} \\
	&= d\left[ \ln \det \left(m^{\alpha\beta}\right) \right],
	\end{split}
	\end{equation*}
	which is exact by Lemma \ref{le:barvinok}. If a different realization is chosen, Theorem \ref{th:density_realization} states that the resulting \eqref{eq:affine_volume} is still solvable.
\end{proof}
\section{Connections with the Nonholonomic Connection}\label{sec:connection}
It turns out that for natural Lagrangian systems subject to \textit{linear} constraints, the divergence - particularly the density form - is encoded in the nonholonomic connection. This interpretation seems to be new.

Throughout this section, let $L:TQ\to\mathbb{R}$ be a natural Lagrangian with Riemannian metric $g$ subject to the \textit{linear} constraints $\eta^\alpha(v)=0$. The nonholonomic connection for this system is given by (cf. \S 5.3 in \cite{bloch2008nonholonomic} and \cite{Vershik}):
\begin{equation*}
\nabla^\mathscr{C}_XY = \nabla_XY + W^i\cdot m_{ij}\left[ X\left(\eta^j(Y)\right) - \eta^j\left(\nabla_XY\right) \right].
\end{equation*}
The equations of motion can then be described via
\begin{equation*}
\nabla^\mathscr{C}_{\dot{q}}\dot{q} = F,
\end{equation*}
where $F$ contains the potential and external forces (the constraint forces are contained in the connection).
\subsection{Torsion}
The nonintegrability of the constraints appears in the torsion of the connection. Computing this, we see
\begin{equation*}
\begin{split}
T^\mathscr{C}(X,Y) &= \nabla^\mathscr{C}_XY - \nabla^\mathscr{C}_YX - [X,Y] \\
&= W^i\cdot m_{ij} \left[ X(\eta^j(Y))-Y(\eta^j(X)) - \eta^j(\nabla_XY-\nabla_YX) \right] \\
&= W^j\cdot m_{ij}\left[ X(\eta^j(Y))-Y(\eta^j(X))-\eta^j([X,Y]) \right] \\
&= W^j\cdot m_{ij} \cdot d\eta^j(X,Y).
\end{split}
\end{equation*}
The torsion can be written as
\begin{equation*}
T^\mathscr{C} = m_{\alpha\beta}\cdot W^\alpha \otimes d\eta^\beta.
\end{equation*}
Indeed, if the constraining 1-forms $\eta^j$ are all closed (so holonomic) then the torsion vanishes. It is worth pointing out that the torsion is vertical-valued; if $X,Y\in M$, then $T^\mathscr{C}(X,Y)\in M^\perp$,i.e. $T(X,Y)$ is orthogonal to the constraint distribution.

Due to the fact that the torsion is a (1,2)-tensor, its trace will be a (0,1)-tensor. Therefore, the trace of the nonholonomic torsion will be a 1-form:
\begin{equation*}
\begin{split}
\tr T^\mathscr{C} &= m_{\alpha\beta}\cdot i_{W^\alpha}d\eta^\beta.
\end{split}
\end{equation*}
Returning to the density form, we see that
\begin{equation*}
\tr T^\mathscr{C} + d\ln\det\left(m^{\alpha\beta}\right) = \vartheta_\mathscr{C},
\end{equation*}
i.e. the trace of the torsion differs from the density form by something exact. This leads to the following theorem.
\begin{theorem}
	A natural nonholonomic system subject to linear constraints has an invariant volume of the form $(\pi_Q^*\rho)\cdot\mu_{\mathscr{C}}$ if and only if there exists functions $\varphi_\gamma$ such that
	\begin{equation*}
	\tr T^\mathscr{C} + \varphi_\gamma\eta^\gamma
	\end{equation*}
	is exact.
\end{theorem}
\begin{remark}
	The vanishing of the torsion shows that the constraints are integrable while the integrability of the (trace of the) torsion shows that a volume is preserved.
\end{remark}
In the case of nonholonomic systems, the nonholonomic connection is compatible with the metric but has nonzero torsion. This idea extends to arbitrary, metric-compatible connections as the following theorem states.
\begin{theorem}
	Let $\tilde{\nabla}$ be an affine connection compatible with the metric with torsion $\tilde{T}$. There exists an invariant volume with density of the form $\pi_Q^*\rho$ for the geodesic spray if and only if $\tr\tilde{T}$ is exact.
\end{theorem}
\begin{proof}
	Consider the volume form on $TQ$ given by
	\begin{equation*}
	\Omega = \det g \cdot dx^1\wedge\ldots\wedge dx^n\wedge dv^1\wedge\ldots\wedge dv^n.
	\end{equation*}
	We want to compute $\mathcal{L}_X\Omega$ where $X$ is the geodesic spray given by
	\begin{equation*}
	X = v^i \frac{\partial}{\partial x^i} - \Gamma^i_{jk}v^jv^k\frac{\partial}{\partial v^i}.
	\end{equation*}
	The Lie derivative is then
	\begin{equation*}
	\begin{split}
	\mathcal{L}_X\Omega &= di_X\Omega \\
	&= \left(d\left[ \det g\right](v) - \det g \left(\Gamma_{ik}^i+\Gamma_{ki}^i\right)v^k\right) \cdot
	\frac{1}{\det g} \Omega,
	\end{split}
	\end{equation*}
	and therefore the divergence is given by
	\begin{equation}\label{eq:torsion_div}
	\mathrm{div}_\Omega(X) = d\left[\ln\det g\right](v) - \left(\Gamma_{ik}^i+\Gamma_{ki}^i\right) v^k.
	\end{equation}
	We will now use the fact that the connection is compatible with the metric:
	\begin{equation*}
	\frac{\partial g_{jk}}{\partial x^i} = g_{\ell k}\Gamma^\ell_{ij} + g_{j\ell}\Gamma_{ik}^\ell.
	\end{equation*}
	This implies that
	\begin{equation*}
	\begin{split}
	g^{jk}\frac{\partial g_{jk}}{\partial x^i} &= g^{jk}g_{\ell k}\Gamma_{ij}^\ell + g^{jk}g_{j\ell}\Gamma_{ik}^\ell \\
	&= \delta_{\ell}^j\Gamma_{ij}^\ell + \delta_{\ell}^k\Gamma_{ij}^\ell = 
	2\Gamma_{ik}^k.
	\end{split}
	\end{equation*}
	Integrating the left-hand side above gives
	\begin{equation}\label{eq:compatable}
	d\left[ \ln\det g\right](v) = 2\Gamma_{ki}^i v^k.
	\end{equation}
	Substituting \eqref{eq:compatable} into \eqref{eq:torsion_div}, we get
	\begin{equation*}
	\mathrm{div}_\Omega(X) = \left( \Gamma_{ki}^i - \Gamma_{ik}^i\right) v^k.
	\end{equation*}
	It remains to show that this is the trace of the torsion. Indeed, 
	\begin{gather*}
	\tilde{T} = \left( \Gamma_{ij}^k - \Gamma_{ji}^k\right) \frac{\partial}{\partial x^k}\otimes dx^i \otimes dx^j \\
	\implies \tr \tilde{T} = \left(\Gamma_{ki}^i - \Gamma_{ik}^i \right) dx^k.
	\end{gather*}
	We conclude that 
	\begin{equation*}
	\mathrm{div}_\Omega(X) = \tr\tilde{T}(v).
	\end{equation*}
\end{proof}
This shows that a way to interpret the torsion of a connection is by measuring how much the geodesic spray fails to preserve volume.

\section{Examples}\label{sec:examples}

We end this work with applying both Theorem \ref{th:main_proved} and Theorem \ref{th:affine_basic} to various nonholonomic systems by calculating the general divergence for nonlinear systems or $\vartheta_{\mathscr{C}}$ and $\zeta_\mathscr{C}$ for affine/linear systems.
Examples are taken from \cite{bloch2008nonholonomic,deLeon1997,monforte2004geometric}.
\subsection{Affine Constraints}
We begin by examining multiple nonholonomic systems subject to affine/linear constraints.
\subsubsection{The Chaplygin Sleigh}\label{sec:sleigh}
As an example of Theorem \ref{th:affine_basic}, we will prove that no invariant basic volumes exist for the Chaplygin sleigh.
The Chaplygin sleigh has the configuration space $Q = \mathrm{SE}_2$, the special Euclidean group, and the following Lagrangian
\begin{equation*}
	L = \frac{1}{2}\left( m\dot{x}^2 + m\dot{y}^2 + \left(I+ma^2\right)\dot{\theta}^2  - 2ma\dot{x}\dot{\theta}\sin\theta + 2ma\dot{y}\dot{\theta}\cos\theta \right),
\end{equation*}
where $(x,y)\in\mathbb{R}^2$ is the coordinate of the contact point, $\theta\in\mathrm{SO}_2$ is its orientation, $m$ is the sleigh's mass, $I$ is the moment of inertia about the center of mass, and $a$ is the distance from the center of mass to the contact point (cf. \S1.7 in \cite{bloch2008nonholonomic}).

The nonholonomic constraint is that the sleigh can only slide in the direction it is pointing and is given by
\begin{equation*}
	\dot{y}\cos\theta - \dot{x}\sin\theta = 0,
\end{equation*}
which corresponds to the 1-form $\eta = \left(\cos\theta\right)dy - \left(\sin\theta\right)dx$ and function $\xi = 0$.

To determine the existence of an invariant volume, we only need to compute $\vartheta_{\mathscr{C}}$ as $\zeta_\mathscr{C}=0$. The constraining vector field and 1-form are:
\begin{equation*}
	W = \frac{ma^2+I}{Im}\left[ \cos\theta \frac{\partial}{\partial y} - \sin\theta\frac{\partial}{\partial x} \right] - \frac{a}{I}\frac{\partial}{\partial\theta}, \quad \eta = (\cos\theta)dy - (\sin\theta)dx.
\end{equation*}
This gives us
\begin{equation*}
	\begin{split}
		\vartheta_{\mathscr{C}} &= \frac{1}{\eta(W)}\mathcal{L}_W\eta \\
		&= \frac{ma}{ma^2+I}\left[(\sin\theta)dy+(\cos\theta)dx\right].
	\end{split}
\end{equation*}
As a consequence of this, the divergence of the Chaplygin sleigh is given by
\begin{equation}\label{eq:divergence_sleigh}
	\mathrm{div}_{\mu_{\mathscr{C}}}(X_H^\mathcal{D}) = -\frac{3mav}{I+ma^2}, \quad
	v = \dot{x}\cos\theta + \dot{y}\sin\theta.
\end{equation}
We want to show that for any $\tilde{\eta}\in\Gamma(\mathcal{D}^0)$, $\vartheta_{\mathscr{C}}+\tilde{\eta}$ is not exact. Because there is only one constraint, it suffices to show that there does not exist a smooth $k$ such that $\vartheta_{\mathscr{C}}+k\cdot\eta$ is closed, i.e. it requires the following to be zero:
\begin{equation*}
	\begin{split}
		d\left( \vartheta_{\mathscr{C}} + k\cdot\eta\right) &= \frac{ma}{ma^2+I}\left[ (\cos\theta)d\theta\wedge dy - (\sin\theta)d\theta\wedge dx \right] \\
		&\quad + \left(\frac{\partial k}{\partial x}\cos\theta + \frac{\partial k}{\partial y}\sin\theta\right) dx\wedge dy \\
		&\quad + \left(\frac{\partial k}{\partial\theta}\cos\theta - k\sin\theta\right) d\theta\wedge dy \\
		&\quad -\left(\frac{\partial k}{\partial\theta}\sin\theta + k\cos\theta\right) d\theta\wedge dx.
	\end{split}
\end{equation*}
Separating the above, we need the following three to vanish:
\begin{equation}\label{eq:compatibility_sleigh}
	\begin{split}
		0 &= \frac{\partial k}{\partial x}\cos\theta + \frac{\partial k}{\partial y}\sin\theta, \\
		0 &= \frac{\partial k}{\partial\theta}\cos\theta - k\sin\theta + \frac{ma}{ma^2+I}\cos\theta, \\
		0 &= \frac{\partial k}{\partial\theta}\sin\theta + k\cos\theta + \frac{ma}{ma^2+I}\sin\theta.
	\end{split}
\end{equation}
The second two lines of \eqref{eq:compatibility_sleigh} are overdetermined for $k$ in the $\theta$-direction and are inconsistent (unless $a=0$ and we obtain the trivial solution $k\equiv 0$). Therefore, there does not exist a smooth $k$ such that $\vartheta_{\mathscr{C}}+k\cdot\eta$ is closed. We note that this is compatable with the known result that when $a=0$, no asymptotically stable dynamics occur.
\subsubsection{The Falling Rolling Disk}
The next example is that of the falling rolling disk whose configuration space is $Q = \mathrm{SE}_2\times S^1$. Its Lagrangian is
\begin{equation*}
	L = \frac{m}{2}\left[ \left(\xi - R\left(\dot\varphi\sin\theta+\dot\psi\right)\right)^2 + \eta^2\sin^2\theta + \left(\eta\cos\theta + R\dot\theta\right)^2\right],
\end{equation*}
where
\begin{equation*}
	\xi = \dot{x}\cos\varphi + \dot{y}\sin\varphi + R\dot\psi, \quad \eta = -\dot{x}\sin\varphi + \dot{y}\cos\varphi.
\end{equation*}
We will consider both the case of the rolling coin on a stationary table (linear constraints) and on a rotating table (affine constraints).
\paragraph{\textbf{Stationary Table}}
The case of the disk on a stationary table was studied in \cite{bbm_disk}.
The constraints are given by the vanishing of the following 1-forms:
\begin{equation}\label{eq:falling_disk}
	\begin{split}
		\eta^1 &= \cos\varphi\cdot dx + \sin\varphi\cdot dy + R\cdot d\psi, \\
		\eta^2 &= -\sin\varphi\cdot dx + \cos\varphi\cdot dy.
	\end{split}
\end{equation}
The corresponding dual vector fields are
\begin{equation*}
	\begin{split}
		W^1 &= \frac{1}{m}\cos\varphi\frac{\partial}{\partial x} + \frac{1}{m}\sin\varphi\frac{\partial}{\partial y} + \frac{R}{I}\frac{\partial}{\partial \psi}, \\
		W^2 &= \frac{J+mR^2}{Jm+m^2R^2\sin^2\theta}\left[ -\sin\varphi\frac{\partial}{\partial x} + \cos\varphi\frac{\partial}{\partial y} \right] - \frac{R\cos\theta}{Jm+m^2R^2\sin^2\theta}\frac{\partial}{\partial\theta}.
	\end{split}
\end{equation*}
Computing $\vartheta_{\mathscr{C}}$, we obtain
\begin{equation*}
	\vartheta_{\mathscr{C}} = -\frac{mR^2\sin(2\theta)}{J+mR^2\sin^2\theta},
\end{equation*}
which is exact. The resulting invariant volume is
\begin{equation*} 
		\frac{1}{J+mR^2\sin^2\theta}\mu_{\mathscr{C}}. 
\end{equation*}
\paragraph{\textbf{Rotating Table}}\label{ex:rotating_disk}
Although the case of the stationary table has been studied, the authors are unaware of any results for the case of a rotating table.
If the falling rolling disk is placed on a table with constant angular velocity $\Omega$, the constraints become affine with
\begin{equation*}
	\begin{split}
		\eta^1(v) + \xi^1 &= 0, \quad \xi^1 = \Omega\left( y\cos\varphi - x\sin\varphi\right) \\
		\eta^2(v) +\xi^2 &= 0, \quad \xi^2 = - \Omega(x\cos\varphi+y\sin\varphi)
	\end{split}
\end{equation*}
where $\eta^\alpha$ are from \eqref{eq:falling_disk}. The volume from the stationary table is still preserved as
\begin{equation*}
	\mathcal{L}_{W^1}\xi^1 = \mathcal{L}_{W^2}\xi^2 = 0.
\end{equation*}
\subsubsection{The Chaplygin Sphere}
We next consider the case of a non-homogeneous sphere rolling without slipping on a horizontal plane, both stationary and rotating. The center of mass of the sphere is located at its geometric center while its principal moments of inertia are distinct. This example has been studied by Chaplygin \cite{chaplygin1903} and, e.g. \cite{bbm2018,kozlov1985,schneider2002}.

The Lagrangian is the kinetic energy,
\begin{equation*}
	\begin{split}
		L &= \frac{1}{2}I_1 \left(\dot{\theta}\cos\psi + \dot{\varphi}\sin\psi\sin\theta\right)^2 + \frac{1}{2}I_2 \left( -\dot{\theta}\sin\psi + \dot{\varphi}\cos\psi\sin\theta\right)^2 \\
		&\quad + \frac{1}{2}I_3\left(\dot{\psi}+\dot{\varphi}\cos\theta\right)^2 + 
		\frac{1}{2}M\left(\dot{x}^2+\dot{y}^2\right),
	\end{split}
\end{equation*}
where it is assumed that the radius is 1.
\paragraph{\textbf{Stationary Table}}
When the table is stationary, the constraints are given by the two 1-forms
\begin{equation*}
	\begin{split}
		\eta^1 &= dx - \sin\varphi\cdot d\theta + \cos\varphi\sin\theta\cdot d\psi, \\
		\eta^2 &= dy + \cos\varphi\cdot d\theta + \sin\varphi\sin\theta \cdot d\psi.
	\end{split}
\end{equation*}
The density form is given by $\vartheta_{\mathscr{C}}=A\cdot d\theta + B\cdot d\psi$ where
\begin{equation*}
	\begin{split}
		A &= \frac{M\sin(2\theta)\left[J_1+J_2\sin^2\psi \right]} {
			2\left( 
			J_3 + J_4\sin^2\theta + J_5\sin^2\theta \sin^2\psi
			\right)}, \\
		B &= \frac{MJ_2\sin(2\psi)\sin^2\theta} {
			2\left( 
			J_3 + J_4\sin^2\theta + J_5\sin^2\theta \sin^2\psi
			\right)},
	\end{split}
\end{equation*}
and the constants $J_j$ are 
\begin{equation*}
	\begin{split}
		J_1 &= I_1I_2-I_1I_3+I_2M-I_3M, \\
		J_2 &= I_1I_3 - I_2I_3 + I_1M - I_2M, \\
		J_3 &= I_3M^2  + I_1I_2I_3 + I_1I_3M + I_2I_3M, \\
		J_4 &= I_2M^2 -I_3M^2 + I_1I_2M-I_1I_3M, \\
		J_5 &= I_1M^2 - I_2M^2 + I_1I_3M - I_2I_3M.
	\end{split}
\end{equation*}
The density form is exact and the resulting invariant volume is
\begin{equation*}
	\begin{gathered}
	\left(\frac{1+\beta}{1-\beta}\right)^{\frac{1}{2}} \cdot\cos\theta \cdot \left(J_3 + (J_3+J_4)\tan^2\theta\right)^{\frac{1}{2}}\mu_{\mathscr{C}}, \\
	\beta = \frac{J_5\sin^2\theta\sin^2\psi}{2J_3+2J_4\sin^2\theta + J_5\sin^2\theta\sin^2\psi}.
	\end{gathered}
\end{equation*}
\paragraph{\textbf{Rotating Table}}\label{ex:rotating_sphere}
When the table is rotating, the constraints becomes affine
\begin{equation*}
	\begin{split}
		\eta^1(v) + \xi^1 &= 0, \quad \xi^1 = \Omega y, \\
		\eta^2(v) + \xi^2 &= 0, \quad \xi^2 = -\Omega x.
	\end{split}
\end{equation*}
Notice that,
\begin{equation*}
	\mathcal{L}_{W^1}\xi^1 = \mathcal{L}_{W^2}\xi^2 = 0, \quad \mathcal{L}_{W^2}\xi^1 = -\mathcal{L}_{W^1}\xi^2 = \frac{\Omega}{M}.
\end{equation*}
As the matrix $(m_{\alpha\beta})$ is symmetric, we have that the product
\begin{equation*}
	m_{\alpha\beta}\mathcal{L}_{W^\alpha}\xi^\beta = 0.
\end{equation*}
Therefore, the Chaplygin sphere on a rotating table is volume-preserving with the same volume as in the stationary case.
\subsubsection{The M\"{o}bius Strip}\label{sec:mobius}
Theorem \ref{th:main_proved} does not require that $Q$ be orientable. Consider the M\"{o}bius strip immersed in $\mathbb{R}^3$ by
\begin{equation}\label{eq:mobius_immersion}
	\begin{split}
		x &= \left( 1 + v\cdot \cos\left( \frac{u}{2}\right)\right) \cdot \cos(u), \\
		y &= \left( 1 + v\cdot \cos\left(\frac{u}{2}\right)\right)\cdot \sin(u), \\
		z &= v\cdot\sin\left(\frac{u}{2}\right),
	\end{split}
\end{equation}
for $0\leq u\leq 2\pi$ and $-1/2<v<1/2$. The Euclidean metric pulled back to the M\"{o}bius strip is
\begin{equation*}
	g = \left( 4v\cos\left(\frac{u}{2}\right) + 2v^2\cos\left(\frac{u}{2}\right) + \frac{v^2}{4} + 2\right) du\otimes du + 2dv\otimes dv.
\end{equation*}
Notice that the above metric is for the double cover rather than the M\"{o}bius strip itself as a consequence of \eqref{eq:mobius_immersion} only being an immersion. However, the resulting equations of motion will be well-defined on the strip.
In two-dimensional systems, any single constraint is automatically holonomic which will make volume-preservation trivial. Let us ``thicken'' the strip by $w$ with resulting metric
\begin{equation*}
	g_{thick} = \left( 4v\cos\left(\frac{u}{2}\right) + 2v^2\cos\left(\frac{u}{2}\right) + \frac{v^2}{4} + 2\right) du\otimes du + 2dv\otimes dv + dw\otimes dw.
\end{equation*}
Consider the linear nonholonomic constraint
\begin{equation*}
	\eta = dv + \sin(u)\cdot dw, \quad W = \frac{1}{2}\frac{\partial}{\partial v} + \frac{1}{2}\sin(u)\frac{\partial}{\partial w}.
\end{equation*}
The density form is
\begin{equation*}
	\vartheta_\mathscr{C} = \frac{\sin(u)\cos(u)}{1+\sin^2(u)}\, du,
\end{equation*}
which is exact. This produces the invariant volume
\begin{equation*}
		\sqrt{ 1 + \sin^2(u)}\cdot \mu_{\mathscr{C}},
\end{equation*}
which is shown in Figure \ref{fig:mobius}.
\begin{figure}
	\centering
	\includegraphics[scale=0.6]{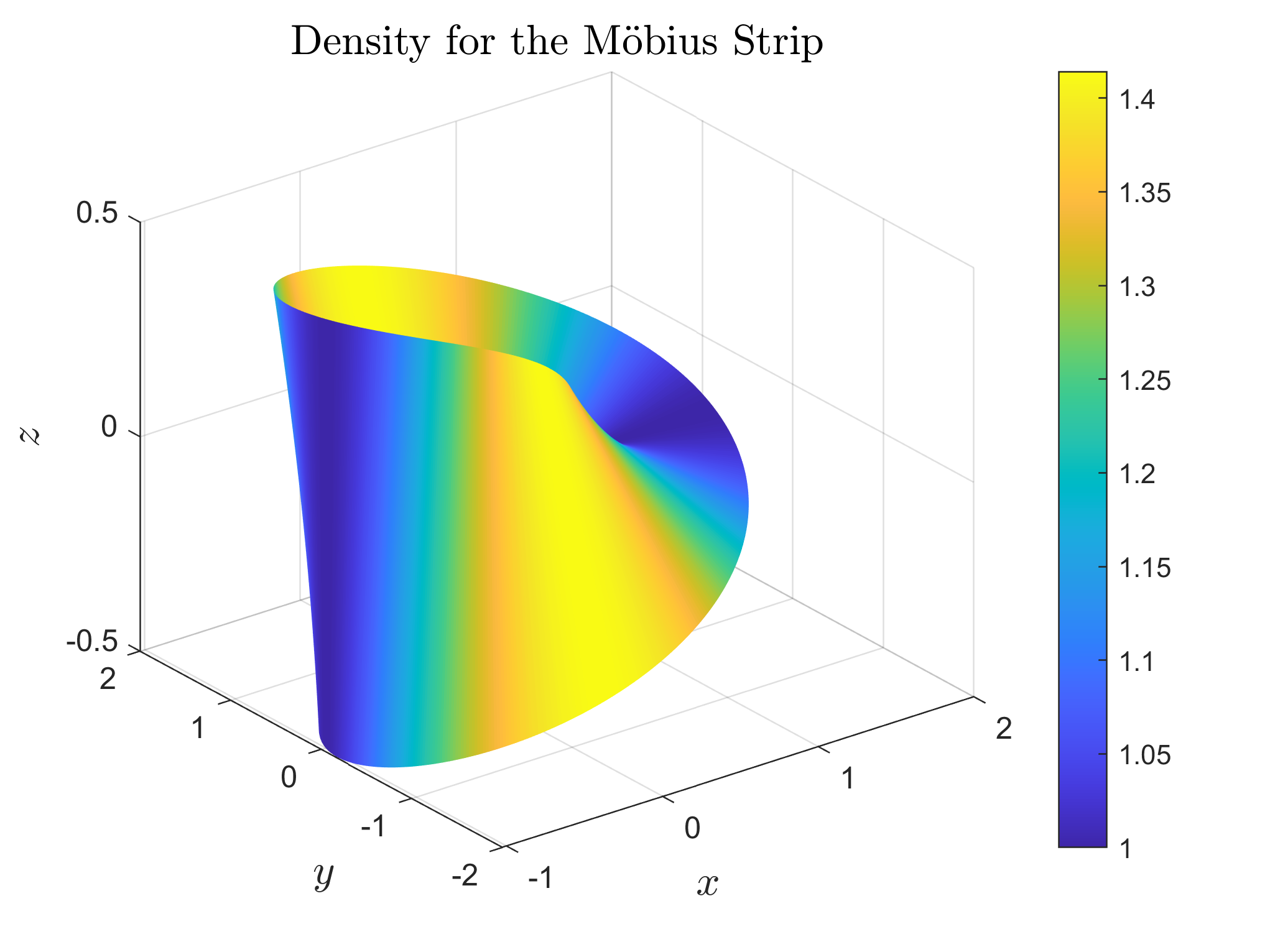}
	\caption{A plot of the density corresponding to the invariant volume for the nonholonomic system on the M\"{o}bius strip.}
	\label{fig:mobius}
\end{figure}
\subsection{Nonlinear Constraints}
The previous examples consisted of a natural Lagrangian and affine constraints which were describable by Theorem \ref{th:affine_basic}. The next two examples utilize Theorem \ref{th:main_proved} on systems with natural Lagrangians but nonlinear constraints. These examples can be found in \cite{deLeon1997,rojo2009}.
\subsubsection{Constant Kinetic Energy}\label{ex:constant_KE}
Let $L:T\mathbb{R}^3\to\mathbb{R}$ be the Lagrangian given by
\begin{equation*}
	L = \frac{1}{2}m\left( \dot{x}^2 + \dot{y}^2 + \dot{z}^2\right) - mgz,
\end{equation*}
subject to the nonlinear constraint of constant kinetic energy
\begin{equation*}
	\Psi = \dot{x}^2 + \dot{y}^2 + \dot{z}^2 - c = 0, \quad c>0.
\end{equation*}
Transferring to the Hamiltonian side, we have
\begin{equation*}
	H = \frac{1}{2m}\left( p_x^2 + p_y^2 + p_z^2\right) + mgz, \quad \Phi = \frac{1}{2}\left( p_x^2 + p_y^2 + p_z^2\right) - c = 0,
\end{equation*}
where $\Phi$ is normalized. As the Lagrangian is natural, \eqref{eq:natural_divergence} can be used to determine the divergence. The requisite data to compute the divergence is
\begin{equation*}
	\begin{split}
		\mathcal{C}^*d\Phi &= m\left( p_xdx + p_ydy + p_zdy\right), \\
		\mathcal{C}^*d\Phi(X_\Phi) &= m\left( p_x^2+p_y^2+p_z^2\right), \\
		\Delta_\mathcal{C}\Phi &= 3m, \\
		\phi &= \frac{1}{2m} - \frac{c}{m(p_x^2+p_y^2+p_z^2)}, \\
		\{H,\phi\} &= \frac{2gcp_z}{(p_x^2+p_y^2+p_z^2)}.
	\end{split}
\end{equation*}
Although $\mathcal{C}^*d\Phi(X_\Phi)\ne 0$ everywhere, it does in some tubular neighborhood of $M$.
The divergence of the system is given by
\begin{equation*}
	\begin{split}
		\mathrm{div}_{\mu_{\mathscr{C}}}(X_H^M) &= -3\cdot \mathcal{C}^*d\Phi\left( [X_H,X_\phi]\right) - 3\{H,\phi\}\mathcal{M}\\
		&= -\frac{18mgcp_z}{(p_x^2+p_y^2+p_z^2)^2} - \frac{18mgcp_z}{(p_x^2+p_y^2+p_z^2)^2} \\
		&= -\frac{9mgp_z}{c},
	\end{split}
\end{equation*}
where we used the fact that $p_x^2+p_y^2+p_z^2=2c$. The divergence does not vanish so $\mu_{\mathscr{C}}$ is not preserved. However, there exists an exact 1-form which produces the divergence:
\begin{equation*}
	\frac{9mg}{c}dz\left( X_H^M\right) = -\mathrm{div}_{\mu_{\mathscr{C}}}(X_H^M).
\end{equation*}
Therefore, the following volume-form is preserved
\begin{equation*}
	\exp\left( \frac{9mgz}{c} \right)\cdot \mu_{\mathscr{C}}.
\end{equation*}
\subsubsection{Appel's Example}
The other nonlinear constraint example we will examine is Appel's example.
The Lagrangian is the same as the constant kinetic energy case,
\begin{equation*}
	L = \frac{1}{2}m\left( \dot{x}^2+\dot{y}^2+\dot{z}^2\right) - mgz,
\end{equation*}
while the nonlinear constraint is now
\begin{equation*}
	\Psi = a^2(\dot{x}^2+\dot{y}^2) - \dot{z}^2 = 0.
\end{equation*}
The data on the Hamiltonian side becomes
\begin{equation*}
	H = \frac{1}{2m}\left( p_x^2+p_y^2 + p_z^2\right) + mgz, \quad \Phi = \frac{a^2}{2}\left(p_x^2+p_y^2\right) - \frac{1}{2}p_z^2 = 0,
\end{equation*}
where, again, the constraint is normalized.
As before, the Lagrangian is natural so \eqref{eq:natural_divergence} can be used. The requisite data is
\begin{equation*}
	\begin{split}
		\mathcal{C}^*d\Phi &= a^2mp_xdx + a^2mp_ydy - mp_zdz, \\
		\mathcal{C}^*d\Phi(X_\Phi) &= ma^4(p_x^2+p_y^2) +mp_z^2, \\
		\Delta_\mathcal{C}\Phi &= m(2a^2-1),\\
		\phi &= \frac{a^2(p_x^2+p_y^2)-p_z^2}{2ma^4(p_x^2+p_y^2)+2mp_z^2},\\
		\{H,\phi\} &= - \frac{a^2(a^2+1)gp_z(p_x^2+p_y^2)}{(a^4p_x^2+a^4p_y^2+p_z^2)^2}.
	\end{split}
\end{equation*}
The constraint is not admissible as $\mathcal{C}^*d\Phi(X_\Phi)$ vanishes at $p_x=p_y=p_z=0$ which is in the constraint manifold. As long as $a\ne 1$, this is the only place where this degeneracy occurs.
The divergence is
\begin{equation*}
	\begin{split}
		\mathrm{div}_{\mu_{\mathscr{C}}}(X_H^M) &= -3\cdot\mathcal{C}^*d\Phi\left( [X_H,X_\phi]\right) - 3\{H,\phi\}\mathcal{M} \\
		&= \frac{12a^2(a^2+1)mgp_z(p_x^2+p_y^2)(a^6p_x^2+a^6p_y^2-p_z^2)}{(a^4p_x^2+a^4p_y^2+p_z^2)^3}.
	\end{split}
\end{equation*}
To simplify the divergence, notice that the constraint makes $p_z^2 = a^2(p_x^2+p_y^2)$. Substituting this, the divergence becomes
\begin{equation*}
	\mathrm{div}_{\mu_{\mathscr{C}}}(X_H^M) = \frac{12(a^2-1)mg}{(a^2+1)p_z}.
\end{equation*}
The following exact 1-form solves the cohomology equation
\begin{equation*}
	\frac{12(a^2-1)}{a^2p_z} dp_z\left(X_H^M\right) = -\mathrm{div}_{\mu_{\mathscr{C}}}(X_H^M),
\end{equation*}
as
\begin{equation*}
	\dot{p}_z = -\frac{a^2mg}{1+a^2}.
\end{equation*}
Therefore, the following form is preserved
\begin{equation*}
	p_z^K\cdot\mu_{\mathscr{C}}, \quad K = \frac{12(a^2-1)}{a^2}.
\end{equation*}
Unfortunately, this form is not a volume-form as it vanishes when $p_z=0$.
\subsection{Non-Mechanical Lagrangians}
All of the examples examined so far have been for systems whose Lagrangian is natural. We conclude this work with two examples whose Lagrangian is not given by a Riemannian metric.
\subsubsection{Higher-Order Lagrangian}
Consider the Lagrangian
\begin{equation*}
	L = \frac{1}{4}\left( \dot{x}^4 + \dot{y}^4 + \dot{z}^4 \right),
\end{equation*}
subject to the nonintegrable constraint $\dot{z} = x\dot{y}$.
The Hamiltonian is
\begin{equation*}
	H = \frac{3}{4}\left( p_x^{4/3} + p_y^{4/3} + p_z^{4/3} \right),
\end{equation*}
and the constraint becomes
\begin{equation*}
	\Phi = p_z^{1/3} - xp_y^{1/3}.
\end{equation*}
Notice that singularities appear when the momentum vanishes. Continuing with the computations, we have
\begin{equation*}
	\begin{split}
		\mathcal{C}^*d\Phi &= dz - xdy \\
		\mathcal{C}^*d\Phi(X_\Phi) &= \frac{x^2}{3p_y^{2/3}} + \frac{1}{3p_z^{2/3}} \\
		\phi &= 3p_z - \frac{3x^2p_z + 3xp_y}{x^2 + \nu^{2/3}}, \quad \nu = \frac{p_y}{p_z} \\
		\Delta_\mathcal{C}\Phi &= 0 		
	\end{split}
\end{equation*}
Applying the constraint, the divergence simplifies to
\begin{equation*}
	\begin{split}
		\mathrm{div}_{\mu_\mathscr{C}}(X_H^M) &= -3p_x^{1/3}\cdot \frac{x^4-2}{x(1+x^4)} \\
		&= -3\frac{x^4-2}{x(1+x^4)}dx(X_H^M)
	\end{split}
\end{equation*}
The following form is exact
\begin{equation*}
	\alpha = \frac{3x^4-6}{x(1+x^4)} dx,
\end{equation*}
and an invariant form is
\begin{equation*}
	\frac{(1+x^4)^{9/4}}{x^6}\mu_\mathscr{C}.
\end{equation*}
Notice the singularity at $x=0$.
\subsubsection{Relativistic Lagrangian}
Suppose we have the relativistic Lagrangian
\begin{equation*}
	L = -m_0c^2\sqrt{ 1 - \frac{v^2}{c^2}}, \quad v^2 = \dot{x}^2 + \dot{y}^2 + \dot{z}^2.
\end{equation*}
With the same constraints $\dot{z} = x\dot{y}$. The Hamiltonian and constraint are
\begin{equation*}
	H = c\sqrt{ p^2 + m^2c^2}, \quad
	\Phi = \frac{c\left( p_z - xp_y\right)}{\sqrt{p^2+m^2c^2}}.
\end{equation*}
The computation yields
\begin{equation*}
	\begin{split}
		\mathcal{C}^*d\Phi &= dz - xdy \\
		\mathcal{C}^*d\Phi(X_\Phi) &= \frac{c\left( (m^2c^2+p_x^2) (1+x^2) +(xp_z + p_y)^2\right)}{\left( m^2c^2 + p^2\right)^{3/2}} \\
		\phi &= \frac{ (p^2+m^2c^2)(p_z - xp_y)}{(m^2c^2+p_x^2) (1+x^2) +(xp_z + p_y)^2} \\
		\Delta_\mathcal{C}\Phi &= 0
	\end{split}
\end{equation*}
Applying the constraints, the divergence simplifies to
\begin{equation*}
	\begin{split}
		\mathrm{div}_{\mu_\mathscr{C}}(X_H^M) &= \frac{3xp_x}{(1+x^2)\sqrt{m^2c^2 + p_x^2 + (1+x^2)p_y^2}} \\
		&= \frac{3x}{1+x^2}dx(X_H^M)
	\end{split}
\end{equation*}
Therefore, an invariant volume is given by
\begin{equation*}
	(1+x^2)^{-3/2}\mu_\mathscr{C}.
\end{equation*}


\section*{Acknowledgments}
We thank Dr. J.C. Marrero for pointing us to existing work in this field as well as the reviewers for helpful comments.


\medskip


\begin{thebibliography}{99}

\bibitem{abraham2008foundations} [10.1090/chel/364]
	\newblock R. Abraham and J.E. Marsden,
	\newblock \emph{Foundations of Mechanics},
	\newblock AMS Chelsea publishing. AMS Chelsea Pub./American Mathematical Society, 2008.

\bibitem{BaGa2012} [10.1007/s00205-012-0512-9]
	\newblock P. Balseiro and L. Garcia-Naranjo,
	\newblock Gauge transformations, twisted Poisson brackets and
        hamiltonization of nonholonomic systems
	\newblock \emph{Arch. Rational. Mech. Anal}, \textbf{205} (2012), 267-310.
	
\bibitem{bbm_disk}[10.1134/S1560354718060035]
	\newblock I.A. Bizyaev, A.V. Borisov, and I.S. Mamaev,
	\newblock An Invariant Measure and the Probability of a Fall in the Problem of an Inhomogeneous Disk Rolling on a Plane,
	\item \emph{Regular and Chaotic Dynamics}, \textbf{23} (2018), 665-684.

\bibitem{bbm2018} [10.1134/S1061920818040027]
	\newblock I.A. Bizyaev, A.V. Borisov, and I.S. Mamaev,
	\newblock Dynamics of the Chaplygin ball on a rotating plane
	\newblock \emph{Russian Journal of Mathematical Physics}, \textbf{25} (2018), 423-433.

\bibitem{BoBi2015}[10.1134/S1061920815040032]
	\newblock I.A. Bizyaev, A.V. Borisov, and I.S. Mamaev,
	\newblock Hamiltonization of elementary nonholonomic systems,
	\newblock \emph{Russian Journal of Mathematical Physics}, \textbf{22} (2015), 444-453.
	
\bibitem{blackall1941}[10.2307/2371286]
	\newblock C.J. Blackall,
	\newblock On volume integral invariants of non-holonomic dynamical systems,
	\newblock \emph{Am. J. Math.}, \textbf{63} (1941), 155--168.
	
\bibitem{bloch2008nonholonomic}[10.1007/978-1-4939-3017-3]
	\newblock A.M. Bloch, J. Baillieul, P. Crouch, J.E. Marsden, D. Zenkov, P.S. Krishnaprasad, and R.M. Murray,
	\newblock \emph{Nonholonomic Mechanics and Control}
	\newblock Springer New York, 2015

\bibitem{bloch2009quasivelocities}[10.1080/14689360802609344]
	\newblock A.M. Bloch, J.E. Marsden and D.V. Zenkov,
	\newblock Quasivelocities and symmetries in non-holonomic systems
	\newblock \emph{Dynamical systems}, \textbf{24} (2009), 187--222.
	
\bibitem{BoBoMa2011} [10.1134/S1560354711050030]
	\newblock A.V. Bolsonov, A.V.  Borisov and I.S. Mamaev
	\newblock Hamiltonization of nonholonomic systems in the neighborhood
        of invariant manifolds
	\newblock \emph{Reg. Chaotic Dynamics}, \textbf{15} (2011), 443-464.

\bibitem{AvBor2005}[10.1007/s11006-005-0085-0]
	\newblock A.V. Bolsinov and I.S. Mamaev,
	\newblock The Nonexistence of an Invariant Measure for an Inhomogeneous Ellipsoid Rolling on a Plane,
	\newblock \emph{Mathematical Notes}, \textbf{77} (2005), 855-857.

\bibitem{AvBor2013}[10.1134/S1560354713030064]
	\newblock A.V. Borisov, I.S. Mamaev, and I.A. Bizyaev,
	\newblock The Hierarchy of Dynamics of a Rigid Body Rolling without Slipping and Spinning on a Plane and a Sphere,
	\newblock \emph{Regular and Chaotic Dynamics}, \textbf{18} (2013), 266-328.

\bibitem{clarkthesis}
	\newblock W. Clark,
	\newblock \emph{Invariant Measures, Geometry, and Control of Hybrid and Nonholonomic Dynamical Systems},
	\newblock Ph.D. thesis, University of Michigan, 2020.

\bibitem{cantrijn2002}[10.1017/S0305004101005679]
	\newblock F. Cantrijn, and J. Cort\'{e}s, and M. de Le\'{o}n, and M. de Diego,
	\newblock On the geometry of generalized Chaplygin systems,
	\newblock \emph{Math. Proc. Camb. Phil. Soc.}, \textbf{132} (2002), 323--351.
	
\bibitem{chaplygin1903}
	\newblock S.A. Chaplygin,
	\newblock On a rolling of a sphere on a horizontal plane,
	\newblock \emph{Mathematical Collection of the Moscow Mathematical Society}, \textbf{24} (1903), 139--168 (Russian).
	
\bibitem{fgn2018}
	\newblock F. Fass\`{o}, L.C. Garc\'{i}a-Narango, and N. Sansonetto,
	\newblock Moving energies as first integrals of nonholonomic systems with affine constraints,
	\newblock \emph{Nonlinearity} \textbf{31} (2018), 755-782.

\bibitem{fedorovkozloz} (MR1351030) [10.1090/trans2/168]
	\newblock Y.N. Fedorov and V.V. Kozlov,
	\newblock Various aspects of n-dimensional rigid body dynamics,
	\newblock in \emph{Dynamical Systems in Classical Mechanics} (eds. V. V. Kozlov) American Mathematical Society Translations: Series 2, (1995), 141--171.
	
\bibitem{federovnaranjo}[10.1007/s00332-014-9227-4]
	\newblock Y.N. Federov, L.C. Garc\'{i}a-Naranjo, and J.C. Marrero,
	\newblock Unimodularity and Preservation of Volumes in Nonholonomic Mechanics,
	\newblock \emph{Journal of Nonlinear Science}, \textbf{25} (2015), 203--246.

\bibitem{fernandez2009hamiltonization}
	\newblock O.E. Fernandez,
	\newblock \emph{The {H}amiltonization of nonholonomic systems and its applications},
	\newblock Ph.D. thesis, University of Michigan, 2009.

\bibitem{forni1995cohomological}[10.1090/S1079-6762-95-03005-8]
	\newblock G. Forni,
	\newblock The cohomological equation for area-preserving flows on compact surfaces,
	\newblock \emph{Electronic Research Announcements of the American Mathematical Society}, \textbf{1} (1995), 114--123.
	
\bibitem{nonbasicaffine}[10.1016/j.physleta.2014.06.026]
	\newblock L.C. Garc\'{i}a-Naranjo, and A.J. Maciejewski, and J.C. Marrero, and M. Przybylska,
	\newblock The inhomogeneous {S}uslov problem,
	\newblock \emph{Physics Letters A}, \textbf{378} (2014), 2389--2395.
	
\bibitem{nonexistenceellipsoid}[10.1134/S1560354713040047]
	\newblock L.C. Garc\'{i}a-Naranjo, and J.C. Marrero,
	\newblock Non-existence of an invariant measure for a homogeneous ellipsoid rolling on the plane,
	\newblock \emph{Regular and Chaotic Dynamics}, \textbf{25} (2013), 372--379.
	
\bibitem{garcia2020}[10.1088/1361-6544/ab5c0a]
	\newblock L.C. Garc\'{i}a-Naranjo, and J.C. Marrero,
	\newblock The geometry of nonholonomic Chaplygin systems revisited,
	\newblock \emph{Nonlinearity}, \textbf{33} (2020), 1297--1341.

\bibitem{ILIYEV1985295}[10.1016/0021-8928(85)90026-7]
	\newblock I. Iliyev,
	\newblock On the conditions for the existence of the reducing chaplygin factor,
	\newblock \emph{Journal of Applied Mathematics and Mechanics}, \textbf{49} (1985), 295--301.
	
\bibitem{isidori1995}[10.1007/978-1-84628-615-5]
	\newblock A. Isidori,
	\newblock \emph{Nonlinear Control Systems},
	\newblock Springer-Verlag London. Communications and Control Engineering, 1995.
	
\bibitem{Jovanovic_1998}[10.1088/0305-4470/31/5/011]
	\newblock B. Jovanovic,
	\newblock Non-holonomic geodesic flows on Lie groups and the integrable {S}uslov problem on {SO}(4),
	\newblock \emph{Journal of Physics A: Mathematical and General}, \textbf{31} (1998), 1415--1422.
	
\bibitem{Jovanovic_2019}[10.2298/TAM190322003J]
	\newblock B. Jovanovic,
	\newblock Note on a ball rolling over a sphere: integrable Chaplygin system with an invariant measure without Chaplygin Hamiltonization,
	\newblock \emph{Theoretical and Applied Mechanics}, \textbf{46} (2019), 97--108.

\bibitem{katok1995introduction}[10.1017/CBO9780511809187]
	\newblock A. Katok and B. Hasselblatt,
	\newblock \emph{Introduction to the Modern Theory of Dynamical Systems},
	\newblock Cambridge University Press. Encyclopedia of Mathematics and its Applications, 1995.

\bibitem{Koiller1992}[10.1007/BF00375092]
	\newblock J. Koiller,
	\newblock Reduction of some classical non-holonomic systems with symmetry,
	\newblock \emph{Archive for Rational Mechanics and Analysis}, \textbf{118} (1992), 113--148.

\bibitem{KOON199721}[10.1016/S0034-4877(97)85617-0]
	\newblock W. S. Koon and J. E. Marsden,
	\newblock The {H}amiltonian and {L}agrangian approaches to the dynamics of nonholonomic systems,
	\newblock \emph{Reports on Mathematical Physics}, \textbf{40} (1997), 21--62.
	
\bibitem{kozlov1985}
	\newblock V.V. Kozlov,
	\newblock On the integration theory of the equations in nonholonomic mechanics,
	\newblock \emph{Advances in Mechanics}, \textbf{8} (1985), 86--107.
	
\bibitem{kozlov1988}[10.1007/BF01077727]
	\newblock V.V. Kozlov,
	\newblock Invariant measures of Euler-Poincaré equations on Lie algebras,
	\newblock \emph{Functional Analysis and Its Applications}, \textbf{22} (1988), 58--59.
	
\bibitem{VVKo1988}[10.1070/RD2002v007n02ABEH000203]
	\newblock V.V. Kozlov,
	\newblock On the Integration Theory of Equations of Nonholonomic Mechanics,
	\newblock \emph{Regular and Choatic Dynamics}, \textbf{7} (2002), 161--176.
	
\bibitem{deLeon1997}[10.1007/BF02435796]
	\newblock M. de Le\'{o}n, J.C. Marrero, and D.M. de Diego,
	\newblock Mechanical systems with nonlinear constraints,
	\newblock \emph{International Journal of Theoretical Physics}, \textbf{36} (1997), 979--995.
	
\bibitem{livvsic1972cohomology}[10.1070/im1972v006n06abeh001919]
	\newblock A.N. Liv\v{s}ic,
	\newblock Cohomology of dynamical systems,
	\newblock \emph{Mathematica of the USSR-Izvestiya}, \textbf{6} (1972), 1278--1301.
	
\bibitem{marle}[10.1016/S0034-4877(98)80011-6]
	\newblock C.M. Marle,
	\newblock Various approaches to conservative and nonconservative nonholonomic systems,
	\newblock \emph{Reports on Mathematical Physics}, \textbf{42} (1998), 211-229.
	
\bibitem{marreropoisson} [10.3934/jgm.2010.2.243]
	\newblock J.C. Marrero,
	\newblock Hamiltonian mechanical systems on Lie algebroids, unimodularity and preservation of volumes,
	\newblock \emph{Journal of Geometric Mechanics}, \textbf{2} (2010), 243--263.
	
\bibitem{milnor}[10.1016/S0001-8708(76)80002-3]
	\newblock J. Milnor,
	\newblock Curvatures of left invariant metrics on lie groups,
	\newblock \emph{Advances in Mathematics}, \textbf{21} (1976), 293-329.
	
\bibitem{molina2010equations}
	\newblock M. Molina-Becerra, E. Freire, and J. Vioque,
	\newblock Equations of motion of nonholonomic {H}amiltonian systems
	\newblock Preprint obtained from http://www. matematicaaplicada2. es/data/pdf/1276179170\_1811485430. pdf

\bibitem{monforte2004geometric}[10.1007/b84020]
	\newblock J.C. Monforte,
	\newblock \emph{Geometric, control and numerical aspects of nonholonomic systems},
	\newblock Springer-Verlag Berlin Heidelberg, 2004.

\bibitem{neimark1972dynamics}
	\newblock J.I. Neimark and N.A. Fufaev,
	\newblock \emph{Dynamics of Nonholonomic Systems},
	\newblock American Mathematical Society. Translations of mathematical monographs, 1972.
	
\bibitem{rojo2009}[10.1103/PhysRevE.80.025601]
	\newblock A.G. Rojo and A.M. Bloch,
	\newblock Nonholonomic double-bracket equations and the Gauss thermostat,
	\newblock \emph{Phys. Rev. E}, \textbf{80} (2009).

\bibitem{ruina1998}[10.1016/S0034-4877(98)80006-2]
	\newblock A. Ruina,
	\newblock Nonholonomic stability aspects of piecewise holonomic systems,
	\newblock \emph{Reports on Mathematical Physics}, \textbf{42} (1998), 91--100.
	
\bibitem{schneider2002}[10.1080/02681110110112852]
	\newblock D. Schneider,
	\newblock Nonholonomic Euler-Poincar\'{e} equations and stability in Chaplygin's sphere,
	\newblock \emph{Dynamical Systems}, \textbf{17} (2002), 87--130.

\bibitem{VANDERSCHAFT1994225}[10.1016/0034-4877(94)90038-8]
	\newblock A.J. Van Der Schaft and B.M. Maschke,
	\newblock On the {H}amiltonian formulation of nonholonomic mechanical systems
	\newblock \emph{Reports on Mathematical Physics}, \textbf{34} (1994), 225--233

\bibitem{Vershik}[10.1142/9789812815453_0014]
	\newblock A.M. Vershik  and  L.D. Faddeev,
	\newblock Lagrangian Mechanics in Invariant Form,
	\newblock \emph{Selecta Math. Sov.}, \textbf{4} (1981), 339--350.

\bibitem{weinstein_poisson}[10.1016/S0393-0440(97)80011-3]
	\newblock  A.  Weinstein,
	\newblock The modular automorphism group of a Poisson manifold,
	\newblock \emph{Journal of Geometry and Physics}, \textbf{23} (1997), 379--394.

\bibitem{YoMo2020} [10.1063/1.5145218]
	\newblock A. Yoshida and P. Morrison,
	\newblock \emph{Deformation of Lie-Poisson algebra and chirality},
	\newblock \emph{J. Mathematical Physics}, \textbf{61} (2020), 092901.
        
\bibitem{zenkov2003invariant}[10.1088/0951-7715/16/5/313]
	\newblock D.V. Zenkov and A.M. Bloch,
	\newblock Invariant measures of nonholonomic flows with internal degrees of freedom,
	\newblock \emph{Nonlinearity}, \textbf{16} (2003), 1793--1807.
	
\bibitem{zenkov1997}[10.1080/02681119808806257]
	\newblock D.V. Zenkov, A.M. Bloch, and J.E. Marsden,
	\newblock The Energy-Momentum Method for the Stability of Nonholonomic Systems,
	\newblock \emph{Dynamics and Stability of Systems}, \textbf{13} (1998), 123--165.



%
%
%
%
%
%
%
%
%
%
%

\end{thebibliography}
\end{document}